\documentclass[a4paper,11pt,reqno]{amsart}

\usepackage{amsmath,amssymb, amsthm}
\usepackage{enumitem} 
\usepackage{graphicx} 
\usepackage{url}
\usepackage{xcolor} 

\title{A polynomial-time algorithm to determine (almost) Hamiltonicity of dense regular graphs}

\author{Viresh Patel and Fabian Stroh}

\address{University of Amsterdam, Korteweg-de Vries Institute (KdVI), Amsterdam, The Netherlands.}
\email{\{V.S.Patel,F.J.M.Stroh\}@uva.nl}
\thanks{V. Patel and F. Stroh are supported by the Netherlands Organisation for Scientific Research (NWO) through the Gravitation Programme Networks (024.002.003) and the NWO TOP grant (613.001.601).}

\def\COMMENT#1{}
\def\NEWCHANGE#1{}

\newcommand{\vol}{\text{vol}}
\newcommand{\RN}{\text{RN}}

\newtheorem{theorem}{Theorem}[section]
\newtheorem{thm} {Theorem}
\newtheorem{thm2} {Theorem}
\newtheorem{lemma}[theorem]{Lemma}

\newtheorem{corollary}[theorem]{Corollary}

\newtheorem{proposition}[theorem]{Proposition}

\newtheorem{claim}[theorem]{Claim}

\theoremstyle{definition}

\newtheorem*{remark*}{Remark}
\newtheorem{remark}[theorem]{Remark}
\newtheorem*{problem*}{Problem}

\numberwithin{equation}{section}

\def\COMMENT#1{}

\begin{document}

\maketitle

\vspace{-0.5 cm}

\begin{abstract}
We give a polynomial-time algorithm for detecting very long cycles in dense regular graphs. Specifically, we show that, given $\alpha\in (0,1)$, there exists a $c=c(\alpha)$ such that the following holds: there is a polynomial-time algorithm that, given a $D$-regular graph $G$ on $n$ vertices with $D\geq \alpha n$, determines whether $G$ contains a cycle on at least $n - c$ vertices. 
The problem becomes NP-complete if we drop either the density or the regularity condition.
The algorithm combines tools from extremal graph theory and spectral partitioning as well as some further algorithmic ingredients.
\end{abstract}

\section{Introduction}
\label{sec:intro}

The study of Hamilton cycles in graphs is a classical part of graph theory. It has been studied intensely from structural, extremal and algorithmic perspectives and is especially relevant due to its connection with the travelling salesman problem. A Hamilton cycle in a graph is a spanning cycle, i.e.\ a cycle that contains every vertex of a graph. This paper is concerned with the algorithmic question of determining whether a graph contains an (almost) Hamilton cycle. The Hamiltonicity problem is NP-hard in general \cite{Garey}, and so there is a  lot of interest in understanding the problem for restricted graph classes. In this paper, we will focus on dense graphs, that is graphs in which the minimum degree is linear in the number of vertices. 

Dirac's theorem \cite{dirac1952some} guarantees the existence of a Hamilton cycle in any $n$-vertex graph of minimum degree at least $n/2$, so this immediately gives a (trivial) algorithm to determine existence in such graphs (and its proof also gives a polynomial-time algorithm for finding a Hamilton cycle). On the other hand, for each $\varepsilon>0$, the problem of determining Hamiltonicity in $n$-vertex graphs of minimum degree $(\frac{1}{2} - \varepsilon)n$ is NP-complete \cite{de1999approximation} (see also Proposition~\ref{pr:hard}).
Our main result, given below, shows that the situation is quite different if we also insist the graphs are regular: we show that determining almost Hamiltonicity in dense regular graphs is polynomial-time solvable.
\begin{thm}
\label{thm:main}
	For every $\alpha\in(0,1]$,  there exists $c = c(\alpha) = 100 \alpha^{-2}$ and a (deterministic) polynomial-time algorithm that, given an $n$-vertex $D$-regular graph $G$  with $D\geq \alpha n$ as input, determines whether $G$ contains a cycle on at least $n - c$ vertices. 
Furthermore	there is a (randomised) polynomial-time algorithm to find such a cycle if it exists.	
\end{thm}
Note that the problem of determining the existence of a very long cycle (as in the result above) becomes NP-complete if we drop either the density or the regularity condition on $G$; see Proposition~\ref{pr:hard}. The question of whether Theorem~\ref{thm:main} holds for $c= c(\alpha) = 0$ (i.e. the Hamilton cycle problem) remains open and is discussed in Section~\ref{sec:conc}. Also, see Remark~\ref{rem:runtime} for a discussion of the explicit running time of the algorithm.

Arora, Karger, and Karpinski \cite{AroraSTOC, Arora} initiated the systematic study of NP-hard problems on dense graphs and this continues to be an active area of research.
The closest result to ours (to the best of our knowledge) is an approximation algorithm for the longest cycle problem in dense (not necessarily regular) graphs that is due to Csaba, Karpinski and Krysta \cite{Csaba2002Approx}. For each $\alpha \in (0, 1/2)$, they give a polynomial-time algorithm which, given an $n$-vertex graph $G$ of minimum degree $\alpha n$, finds a cycle of length at least $(\frac{\alpha}{1 - \alpha}) \ell$, where $\ell$ is the length of the longest cycle in $G$.\footnote{The actual approximation ratio here is $(\frac{\alpha}{1 - \alpha}) - \varepsilon$ for arbitrarily small $\varepsilon$. As mentioned, for $\alpha \geq 1/2$, Dirac's theorem gives a trivial algorithm for the longest cycle problem.} They also show one cannot replace $(\frac{\alpha}{1 - \alpha})$ with $(1 - \varepsilon_0(1-2\alpha))$ where $\varepsilon_0 = 1/742$ unless $P = NP$. The two algorithms are not directly comparable: while theirs works on all dense graphs, ours achieves a much better approximation ratio for dense regular Hamiltonian graphs. In Section~\ref{sec:conc}, we discuss how our methods can be used for the longest cycle problem to achieve an approximation ratio very close to one for general dense regular graphs.
%

Our algorithm is inspired by questions and results about Hamiltonicity in extremal graph theory; here one is interested in various types of conditions that guarantee Hamiltonicity such as in Dirac's theorem. 
There are two extremal examples that show $n/2$ is tight in Dirac's theorem: a slightly imbalanced complete bipartite graph and a graph consisting of two disjoint cliques. One might hope to eliminate these barriers to Hamiltonicity by imposing some connectivity and regularity conditions. 
In this direction, Bollob\'as \cite{bollobasconj} and H{\"a}ggkvist (see \cite{jackson1980hamilton}) independently conjectured that a $t$-connected regular graph with degree at least $n/(t+1)$ is Hamiltonian. 
Jackson \cite{jackson1980hamilton} proved the conjecture for $t=2$, while Jung \cite{Jung} and Jackson, Li, and Zhu \cite{JacksonLiZhu} gave an example showing the conjecture does not hold for $t \geq 4$. Finally, K\"uhn, Lo, Osthus, and Staden \cite{kuhn2014robust, kuhn2016solution} resolved the conjecture by proving the case $t=3$ asymptotically.
Although the conjecture does not hold in general, it suggests that questions of Hamiltonicity (and long cycles) might be easier in some sense for (dense) regular graphs, and our result seems to confirm this.



Our algorithm relies heavily on the notion of robust expansion, a notion of expansion for dense (directed) graphs introduced and applied by K\"uhn and Osthus together with several co-authors to resolve and make progress on a number of long-standing conjectures in extremal graph theory; see for example \cite{Rob0,Rob2,Rob4,Robb2}. In particular, K\"uhn, Lo, Osthus and Staden \cite{kuhn2014robust, kuhn2016solution}, in their proof of the $t=3$ case of the Bollob{\'a}s-H{\"a}ggkvist conjecture showed that all dense regular graphs have a vertex partition into a small number of parts where each part induces a (bipartite) robust expander. This decomposition is central to our algorithm, and by combining their argument with some spectral partitioning techniques, we are able to construct such a partition algorithmically in polynomial time; this may be of independent interest.
A further by-product of this is that we can partially answer a question of K{\"u}hn and Osthus~\cite{Robb2} about algorithms to check whether a graph is a robust expander in polynomial time; this result and its background are presented in Section~\ref{sec:digression} after robust expansion has been formally defined.

Once we have the algorithm for constructing the robust expander partition, we will also require a result of Letzter and Gruslys \cite{CyclePartitions} for finding certain structures between the parts in this partition. Combining all of this with some further algorithmic ingredients will yield the desired algorithm.

Below we give a more detailed account of our algorithm as well as the proof of the hardness results (Proposition~\ref{pr:hard}) mentioned above.
In Section~\ref{sec:notation} we give some general notation and we formally define robust expansion, as well stating some of the results from spectral graph theory that we will need in later sections.
In Section~\ref{sec:robustparts}, we give  the algorithm for finding the robust expander partition mentioned above, and Section~\ref{sec:findingcycles} is about utilizing the structure of a robust partition to find a long cycle. This is where the proof of Theorem~\ref{thm:main} is given.

\subsection{Proof outline}

We now go into more detail about our algorithm. 
The first step of the algorithm,  given in Section~\ref{sec:robustparts}, is to obtain a so-called robust partition of our graph. This is a vertex partition in which each part induces a robust expander or a bipartite robust expander and where there are few edges between parts. We give the precise definitions below, but informally we can think of (bipartite) robust expanders as dense (bipartite) graphs with good connectivity properties that are resilient to small alterations. In \cite{kuhn2014robust}, it was shown that such a robust partition exists for dense regular graphs, and crucially, the number of parts is independent of the number of vertices and depends only on the density. The idea of the proof in \cite{kuhn2014robust} is to iteratively refine the vertex partition as follows. Given a vertex partition $\mathcal{P} = \{U_1, \ldots, U_k\}$, if some $U_i$ is not a (bipartite) robust expander, then they show there exists a  partition $U_i = A \cup B$ of $U_i$ where there are few edges between $A$ and $B$;  $U_i$ is then replaced with $A,B$ in $\mathcal{P}$ and this is repeated with the new partition. This process must end after a finite number of steps since the density inside parts increases at each step (since there were not many edges between $A$ and $B$). We follow this argument closely, except that the existence of $A,B$ is not enough for us: we need a polynomial-time algorithm to find $A$ and $B$. We make use of spectral algorithms to achieve this.

In the second step, given in Section~\ref{sec:findingcycles}, we make use of the robust partition to decide whether a very long cycle exists.
Using further results from \cite{kuhn2014robust}, it turns out that a very long cycle exists if and only if a certain type of structure exists between the parts of our robust partition. With the help of a result from \cite{CyclePartitions}, we give a fast algorithm to determine whether such a structure is present in our graph and to find it if it is. We will give a more detailed sketch of this at the start of Section~\ref{sec:findingcycles}.

%


We end this subsection by proving the simple hardness results mentioned earlier in the introduction.

\begin{proposition}
\label{pr:hard}
For each fixed integer $C \geq 0$ and each real $\alpha \in (0, 1/2)$ the following holds.
\begin{itemize}[noitemsep]
\item[{\rm (i)}] The problem of deciding whether a regular $n$-vertex graph has a cycle of length at least $n - C$ is NP-complete.
\item[{\rm (ii)}] The problem of deciding whether an $n$-vertex graph of minimum degree at least $\alpha n$ has a cycle of length at least $n - C$ is NP-complete.
\end{itemize}
\end{proposition}
\begin{proof}
For part (i), it is known that the problem of determining Hamiltonicity of $3$-regular graphs is NP-complete \cite{garey1976planar}. Fix $C$ even with $C \geq 4$.
Given a $3$-regular graph $G$, let $H$ be the disjoint union of $G$ with an arbitrary $3$-regular graph on $C$ vertices and assume $H$ has $n$ vertices. Then $G$ has a Hamilton cycle if and only if $H$ has a cycle of length at least $n-C$ and so a polynomial-time algorithm for the problem in part (i), for $C$ even and at least $4$, would give a polynomial-time algorithm for deciding Hamiltonicity in $3$-regular graphs.

For the remaining cases of $C$, given a $3$-regular graph $G$, consider the $3$-regular graph $G'$ on $3|V(G)|$ vertices obtained from $G$ by replacing each vertex of $G$ with a triangle in such a way that we recover $G$ by contracting each triangle to a vertex. It is not hard to see that the following are equivalent:
\begin{itemize}
\item[] $G$ has a Hamilton cycle;
\item[] $G'$ has a Hamilton cycle;
\item[] $G'$ has a cycle of length $|V(G')| - 1$;
\item[] $G'$ has a cycle of length $|V(G')| - 2$.
\end{itemize}
For $C =1$ and $C \geq 4$ odd, let $H$ be the disjoint union of $G'$ with an arbitrary $3$-regular graph on $C-1$ (even) vertices and for $C=2$ let $H = G'$. Then $H$ has a cycle of length at least $n - C$ if and only if $G$ has a Hamilton cycle. 

(ii) We reduce to the problem of deciding the existence of a Hamilton path in general graphs, which is known to be NP-complete \cite{Garey}. Given a graph $G$ on $k$ vertices, construct the graph $H$ as follows. Start by taking a complete bipartite graph with bipartition $V(H) = A \cup B$ where $|A| = 1 + r$ and $|B| = (C+1)k + r$ and $r$ is chosen so that  $|A| / (|A| + |B|) > \alpha$. Now we insert $C+1$ disjoint copies of $G$ into $B$ to form $H$. Note that $\delta(H) \geq \alpha |V(H)|$ and it is easy to see that $H$ has a cycle of length $|V(H)| - C$ if and only if $G$ has a Hamilton path. This gives the desired reduction since $|V(H)|$ is linear in $|V(G)|$.
\end{proof}

\section{Preliminaries}
\label{sec:notation}

We follow general graph theory notation found e.g.\ in \cite{diestel}.

Given a graph $G$, we denote its vertex and edge sets by $V(G)$ and $E(G)$ respectively. For a vertex $v \in V(G)$, we write $N(v)$ for the neighbours of $v$ in $G$ and write $d(v):=|N(v)|$ for the degree of $v$. Given $S \subseteq V(G)$, we also write $d_S(v):= |N(v) \cap S|$ for the degree of $v$ in $S$. We denote with $\delta(G)$ the smallest degree among vertices in $G$.

We write $H \subseteq G$ to mean that $H$ is a subgraph of $G$, i.e.\ $V(H) \subseteq V(G)$ and $E(H) \subseteq E(G)$.
We define $E_G(S):=\{ab \in E(G) \mid a,b \in S \}$ and
we write $G[S]$ for the graph induced by $G$ on $S$, i.e. the graph with vertex set $S$ and edge set $E_G(S)$.
	For $S,T\subseteq V(G)$, we define $E_G(S,T):=\{xy\in E(G)\mid x\in S, y\in T\}$ and $e_G(S,T):=|E_G(S,T)|$. We will sometimes omit the subscript if it is clear.  For $S,T\subseteq V(G)$ disjoint,  we write $G[S,T] := (S \cup T, E_G(S,T))$ for the bipartite graph induced by $G$ between $S$ and $T$.
We often denote the complement of $S \subseteq V(G)$ by $\overline{S}$ i.e.\ $\overline{S}:= V(G) \setminus S$.

We write $a \ll b$ to mean that $a \leq f(b)$ for some implicitly given non-decreasing function $f: (0,1] \rightarrow (0,1]$.	Informally, this is understood to mean that $a$ is small enough in relation to $b$. We sometimes also write $a \ll_f b$ when we wish to be specific about the function $f$.

\subsection{Spectral partitioning}

Given a graph $G$ and $S \subseteq V(G)$, 
the \textit{conductance} of $S$, written $\Phi(S) = \Phi_G(S)$, is given by 
$$\Phi(S) := \frac{e_G(S,\overline{S})}{\min(\vol_G(S),\vol_G(\overline{S}))},$$
where $\vol_G(S) = \vol(S):=\sum_{i\in S}d(i)$ refers to the volume of $S$.
The \emph{edge expansion} $\Phi(G)$ of $G$ is defined by $\Phi(G) := \min_{S\subseteq V(G)}\Phi(S)$. 

We write $A_G \in \mathbb{R}^{V(G) \times V(G)}$ for the adjacency matrix of $G$, where $A_G$ is the matrix whose rows and columns are indexed by vertices of $G$ and is defined by
\begin{align*}
(A_G)_{uv} := 
\begin{cases}
1 &\text{if } uv \in E(G); \\
0 &\text{otherwise}.
\end{cases}
\end{align*}
We write
$$ L_G := I-D^{-\frac{1}{2}}A_GD^{-\frac{1}{2}} $$ for the normalized Laplacian of $G$, where $I \in \mathbb{R}^{V(G) \times V(G)}$ is the identity matrix and $D$ is the diagonal matrix of degrees (where $D_{uu} = d(u)$ for each $u \in V(G)$ and $D_{uv}=0$ for $u \not= v$). 

Suppose the eigenvalues of $L_G$ are ordered $\lambda_1 \leq \lambda_2 \leq \ldots \leq \lambda_n$.
The following gives an algorithm for approximating the expansion of $G$ and giving a corresponding partition of the vertices.
\begin{theorem}[\cite{alon1986eigenvalues}, \cite{trevisan2012max}] 
	\label{th:cheeger} For any graph $G$, 
	we have $\frac{\lambda_2}{2}\leq \Phi(G) \leq \sqrt{2 \lambda_2}$ and there is an algorithm that finds $S\subseteq V$ such that $\Phi(S) \leq \sqrt{2 \lambda_2}$ in time polynomial in $n=|V(G)|$. In particular, $\Phi(G) \geq \Phi(S)^2/4$. 
\end{theorem}

The inequality
$\frac{\lambda_2}{2}\leq \Phi(G) \leq \sqrt{2 \lambda_2}$ is often referred to as Cheeger's inequality.
There is an analogue of Cheeger's inequality for the largest eigenvalue $\lambda_n$ and the \emph{bipartiteness ratio} $\beta(G)$. For $y\in \{-1,0,1\}^{V(G)}\setminus \{\textbf{0}\}$ we define
$$\beta(y):= \frac{\sum_{uv\in E(G)} |y_u+y_v|}{\sum_{v\in V(G)} d_G(v)  |y_v|}  $$
and $\beta(G):= \min_{y\in \{-1,0,1\}^V\setminus \{\textbf{0}\}} \beta(y)$. We can think of a small value $\beta(G)$ to mean that $G$ is close to bipartite. In particular, if we set $A = \{v \in V(G): y_v = 1\}$ and $B = \{v \in V(G): y_v = -1\}$ then 
\begin{equation}
\label{eq:beta}
\beta(y)= \frac{2e_G(A) + 2e_G(B) + e_G(A\cup B, V(G) \setminus (A \cup B))} {\vol_G(A \cup B)}.
\end{equation}

\begin{theorem}[\cite{trevisanLectureNotes,trevisan2012max}]
	\label{th:trevisan} For any graph $G$, 
	we have $\frac{2-\lambda_n}{2}\leq\beta(G)\leq\sqrt{2(2-\lambda_n)}$ and there is an algorithm that finds $y\in \{-1,0,1\}^{V(G)}$ such that $\beta(y)\leq\sqrt{2(2-\lambda_n)}$ in time polynomial in $n=|V(G)|$. In particular, $\beta(G) \geq \beta(y)^2 /4$
\end{theorem}

\begin{remark}
	The algorithms from both Theorem~\ref{th:cheeger} and \ref{th:trevisan} run in time $O(|E(G)|+|V(G)|\log|V(G)|).$\NEWCHANGE{Added this remark. I looked around, but I couldn't figure out if this includes the time required to compute the Fiedler vector.}
\end{remark}


%

\subsection{Robust expanders}

The following definitions follow closely those in \cite{kuhn2014robust}. Throughout, assume $G$ is an $n$-vertex graph. 


\vspace{0.2 cm}
\noindent
{\bf Robust expanders and bipartite robust expanders} -
Given an $n$-vertex graph $G$, and $S \subseteq V(G)$ and parameters $0<\nu\leq \tau<1$, we define the \emph{$\nu$-robust neighbourhood} of $S$ to be $\RN_{\nu,G}(S):=\{v\in G\mid d_S(v)\geq \nu n\}$. We say $G$ is a \emph{robust $(\nu,\tau)$-expander} if for all $S\subseteq V(G)$ with $\tau n\leq |S| \leq (1-\tau)n$ we have $|\RN_{\nu, G}(S)|\geq |S|+\nu n$. We say $G$ is a \emph{bipartite robust $(\nu,\tau)$-expander} with bipartition $A,B$ if $A,B$ is a partition of $V(G)$ and for every $S\subseteq A$ with $\tau|A|\leq |S|\leq (1-\tau)|A|$ we have $|\RN_{\nu,G}(S)|\geq|S|+\nu n$. Note that the order of $A$ and $B$ matters here. 
	
\vspace{0.2 cm}
\noindent
{\bf Robust expander components and bipartite robust expander components} -
	Given $0<\rho<1$ and an $n$-vertex graph $G$, we say that $U\subseteq V(G)$ is a \emph{$\rho$-component} if $|U|\geq \sqrt{\rho}n$ and $e_G(U,\overline{U})\leq \rho n^2$, where as usual $\overline{U}:=V(G)\setminus U$. We say that $U$ is \emph{$\rho$-close to bipartite} with bipartition $A,B$ if $A,B$ is a partition of $U$, $|A|,|B|\geq \sqrt{\rho}n$, $||A|-|B||\leq \rho n$, and $e_G(A,\overline{B})+e_G(B,\overline{A})\leq \rho n^2$.
	We will sometimes call a graph a $\rho$-component or $\rho$-close to bipartite if this holds for its vertex set.
	We say that $G[U]$ is a \emph{$(\rho,\nu,\tau)$-robust expander component} of $G$ if $U$ is a $\rho$-component and $G[U]$ is a robust $(\nu,\tau)$-expander. We say that $G[U]$ is a \emph{bipartite $(\rho,\nu,\tau)$-robust expander component} with bipartition $A,B$ if $U$ is $\rho$-close to bipartite with bipartition $A,B$ and $G[U]$ is a bipartite robust $(\nu,\tau)$-expander with bipartition $A,B$.

	We now introduce the concept of a robust partition, which is central to our result.

\vspace{0.2 cm}
\noindent	
{\bf Robust partitions}	- 
	Let $k,\ell, D \in \mathbb{N}$ and $0<\rho\leq \nu\leq \tau<1$. Given an $n$-vertex, $D$-regular graph $G$, we say that $\mathcal{V}$ is a \emph{robust partition of $G$ with parameters $\rho,\nu,\tau,k,\ell$} if the following hold: 
	\begin{enumerate}[noitemsep]
		\item[(D1)] $\mathcal{V} = \{V_1,\dots,V_k,W_1,\dots,W_\ell\}$ is a partition of $V(G)$;
		\item[(D2)] for all $1 \leq i \leq k$, $G[V_i]$ is a $(\rho,\nu,\tau)$-robust expander component of $G$;
		\item[(D3)] for all $1 \leq j \leq \ell$, there exists a partition $A_j,B_j$ of $W_j$ such that $G[W_j]$ is a bipartite $(\rho,\nu,\tau)$-robust expander component with bipartition $A_j,B_j$;
		\item[(D4)] for all $X,X'\in \mathcal{V}$ and all $x\in X$, we have $d_X(x)\geq d_{X'}(x)$; in particular, $d_X(x)\geq D/m$, where $m:= k+\ell$;
		\item[(D5)] for all $1\leq j \leq \ell$, we have $d_{B_j}(u)\geq d_{A_j}(u)$ for all $u\in A_j$ and $d_{A_j}(v)\geq d_{B_j}(v)$ for all $v\in B_j$; in particular, $\delta(G[A_j,B_j])\geq D/2m$;
		\item[(D6)] $k+2\ell \leq \lfloor (1+\rho^{1/3})n/D\rfloor$;
		\item[(D7)] for all $X\in \mathcal{V}$, all but at most $\rho n$ vertices $x\in X$ satisfy $d_X(x)\geq D-\rho n$.
	\end{enumerate}

	For technical reasons, we also introduce weak robust subpartitions.	We will use this definition and the following result only in Section~\ref{sec:findingcycles}. A weak robust subpartition differs from a robust partition mainly in that the disjoint subsets need not be a partition of the vertices.
	Let $k,\ell\in \mathbb{N}_0$ and $0<\rho\leq\nu\leq\tau\leq \eta <1$. Given a graph $G$ on $n$ vertices, we say that $\mathcal{U}$ is a \emph{weak robust subpartition} of $G$ with parameters $\rho,\nu,\tau,\eta,k,\ell$ if the following conditions hold:
	\begin{enumerate}[noitemsep]
		\item[(D1$'$)] $\mathcal{U}= \{U_1,\dots,U_k,Z_1,\dots,Z_\ell\}$ is a collection of disjoint subsets of $V(G)$;
		\item[(D2$'$)] for all $1\leq i \leq k$, $G[U_i]$ is a $(\rho,\nu,\tau)$-robust expander component of $G$;
		\item[(D3$'$)] for all $1\leq j \leq \ell$, there exists a partition $A_j,B_j$ of $Z_j$ such that $G[Z_j]$ is a bipartite $(\rho,\nu,\tau)$-robust expander component with bipartition $A_j,B_j$;
		\item[(D4$'$)] $\delta(G[X])\geq \eta n$ for all $X\in \mathcal{U}$;
		\item[(D5$'$)] for all $1\leq j\leq \ell$, we have $\delta(G[A_j,B_j])\geq \eta n/2$.
	\end{enumerate}

	\begin{lemma}[Proposition 6.1 in \cite{kuhn2014robust}]
		\label{lem:weakrobustsubp}
		Let $k,\ell,D\in \mathbb{N}_0$ and suppose that $0< 1/n \ll \rho \leq \nu\leq \tau\leq \eta \leq \alpha^2/2 < 1$. Suppose that $G$ is a $D$-regular graph on $n$ vertices where $D\geq \alpha n$. Let $\mathcal{V}$ be a robust partition of $G$ with parameters $\rho,\nu,\tau,k ,\ell$. Then $\mathcal{V}$ is a weak robust subpartition of $G$ with parameters $\rho,\nu,\tau,\eta,k,\ell$.
	\end{lemma}

\section{Robust partitions}
\label{sec:robustparts}
\subsection{Statements of algorithms}
In this section we present an algorithm (Theorem~\ref{th:decompose}) that we use to find  robust partitions (see previous section for the definition) of regular graphs. As mentioned earlier, the main algorithm and its analysis are obtained by combining the robust expander decomposition of regular graphs from \cite{kuhn2014robust} together with spectral algorithms for graph partitioning from \cite{trevisan2012max, trevisanLectureNotes}.

We begin by presenting four algorithms in the following lemmas that will eventually be used together to obtain the main algorithm. The proofs appear in the following subsection. 


\begin{lemma}
	\label{th:alg1}
	For each fixed choice of parameters $1/n_0 \ll \rho \ll \nu \ll \rho' \ll \tau \ll \alpha < 1$ there exists a polynomial-time algorithm that does the following.
	Given a $D$-regular $n$-vertex graph $G=(V,E)$ and $U \subseteq V$ as input, where $D \geq \alpha n$, $n \geq n_0$ and $G[U]$ is a $\rho$-component of $G$ that is not $\rho'$-close to bipartite, the algorithm determines that either
	\begin{enumerate}[noitemsep]
		\item [(i)] $G[U]$ is a robust $(\nu,\tau)$-expander, or
		\item [(ii)] $U$ has a partition $U_1$, $U_2$ such that $U_1$, $U_2$ are $\rho'$-components,
	\end{enumerate}
	and in the case of (ii) identifies the partition $U_1, U_2$. 		We call this Algorithm~1.
\end{lemma}

\begin{lemma}
	\label{th:alg2}
	For each fixed choice of parameters $1/n_0 \ll \rho \ll \rho'\ll\alpha <1$ there is a polynomial time algorithm that does the following. Given a $D$-regular, $n$-vertex graph $G=(V,E)$ and $U \subseteq V$ as input, where $D \geq \alpha n$, $n \geq n_0$, and $G[U]$ is a $\rho$-component of $G$, the algorithm determines that either
	\begin{enumerate}[noitemsep]
		\item[(i)] $G[U]$ is not $\rho$-close to bipartite, or
		\item[(ii)] $G[U]$ is $\rho'$-close to bipartite,
	\end{enumerate}
and in the case of (ii) identifies the corresponding bipartition. We call this Algorithm~2.
\end{lemma}

\begin{lemma}
	\label{th:alg3}	
	For each fixed choice of parameters $1/n_0 \ll \rho \ll \nu  \ll \rho'\ll \tau \ll \alpha < 1$ there is a polynomial-time algorithm that does the following. Given a   a $D$-regular, $n$-vertex graph $G=(V,E)$ and $U \subseteq V$ as input, where $D \geq \alpha n$, $n \geq n_0$, and $G[U]$ is $\rho$-close to bipartite with bipartition $A,B$, the algorithm determines that either
	\begin{enumerate}[noitemsep]
		\item[(i)] $G[U]$ is a bipartite robust $(\nu,\tau)$-expander with bipartition $A,B$, or
		\item[(ii)] $U$ has a partition $U_1$, $U_2$ such that $G[U_1]$, $G[U_2]$ are $\rho'$-components,
	\end{enumerate}
	and in the case of (ii) identifies the partition $U_1,U_2$ of $U$.
		We call this Algorithm~3.
\end{lemma}

\begin{lemma}
	\label{th:alg6}	
	For each fixed choice of parameters $1/n_0 \ll \rho \ll \nu \ll \rho'\ll \tau \ll \alpha < 1$ there exists a polynomial-time algorithm that does the following. Given  a $D$-regular $n$-vertex graph $G=(V,E)$ and $U \subseteq V$ as input, where $D \geq \alpha n$, $n \geq n_0$, and $G[U]$ is a $\rho$-component, the algorithm determines that either
	
	\begin{enumerate}[noitemsep]
		\item[(i)] $G[U]$ is a robust $(\nu,\tau)$-expander, or
		\item[(ii)] $G[U]$ is a bipartite robust $(\nu,\tau)$-expander, or		
		\item[(iii)] $U$ has a partition $U_1$, $U_2$ such that $G[U_1]$, $G[U_2]$ are $\rho'$-components,
	\end{enumerate}
and in the case of (ii) and (iii) identifies the corresponding partition.	We call this Algorithm~4.
\end{lemma}

\begin{remark}
In each of the four lemmas above, the algorithm distinguishes between various cases. It may be that more than one of these cases hold for the given input graph; if so then the algorithm will output any one case that holds for the given graph.

The running time of each of the algorithms is $O(n^3)$, where $n$ is the number of vertices of the input graph. The running time does not depend at all on the fixed parameters (not even as hidden constants in the `Big O' notation). However in each lemma, the hierarchy is necessary for the fixed parameters in order to guarantee that at least one of the outcomes occurs in the conclusion of the lemma.
\end{remark}

\subsection{Proofs of correctness of algorithms}
We now give the proofs of the preceding lemmas. We begin with a simple proposition.

	
	\begin{proposition}\label{Claim1}
	Let $G$ be an $n$-vertex $D$ regular graph with $D \geq \alpha n$ and let $U$ be a $\rho$-component of $G$. Then
	\begin{itemize}[noitemsep]
	\item[(i)]	$|U| \geq D - \sqrt{\rho}n \geq (\alpha - \sqrt{\rho})n$
	\item[(ii)] There are at most $\frac{2\rho}{\alpha(\alpha-\sqrt{\rho})}|U|$ vertices of
	degree at most $\frac{1}{2}\alpha n$ in $G[U]$.  
	\end{itemize}	
	\end{proposition}
	\begin{proof}
		(i) Since $G$ is $D$-regular and $U$ is a $\rho$-component, we have $ \frac{1}{2}|U|^2 \geq e_G(U) \geq \frac{1}{2}D |U| - \rho n^2,$ from which we obtain
		$ |U| \geq D - \frac{\rho n^2}{|U|}\geq D - \sqrt{\rho}n,$
		where the second inequality uses that $|U|\geq \sqrt{\rho}n$ since it is a $\rho$-component.
		
		(ii) If the number of vertices of degree at most $\frac{1}{2} \alpha n$ is $\gamma |U|$, then we have 
		$$ (D/2)\gamma|U| + D(1-\gamma)|U| \geq 2e_G(U) \geq D|U|- \rho n^2, $$
		from which we get
		$\gamma \leq \frac{2\rho n^2}{ D |U|} \leq  \frac{2\rho}{\alpha(\alpha-\sqrt{\rho})}$ using part (i) and $D \geq \alpha n$ for the final inequality.
	\end{proof}


	\begin{remark}
		\label{rem1}
 A similar calculation shows that if $U$ is $\sigma$-close to bipartite with bipartition $A,B$, we have $|A|,|B|\geq D - 2\sqrt{\sigma}n \geq (\alpha -2\sqrt{\sigma})n$.
	\end{remark}

\begin{proof}[Proof of Lemma~\ref{th:alg1}]

	We will use the algorithm in Theorem~\ref{th:cheeger} to iteratively find subgraphs of $G[U]$ that are not well connected to the rest of $U$ and remove them until this is no longer possible. 
	If this process continues to a point where the removed parts are large enough then we can show both the removed part and the remaining part each form a $\rho'$-component. If the process stops before the removed part becomes large  then we can show $G[U]$ is a robust expander.

Let $G=(V,E)$ and, in this proof, for any subset $S \subseteq U$ we will use $\overline{S}$ to denote $U \setminus S$ rather than our usual convention where it denotes $V \setminus S$. 

Let $n'=|U|$ so that $n' \geq (\alpha - \sqrt{\rho})n \geq \frac{1}{2} \alpha n$ (by the previous proposition).
	Let $U_0$ be the vertices of degree at most $\frac{1}{2}\alpha n$ in $G[U]$ so that $|U_0| \leq  \frac{2\rho}{\alpha(\alpha-\sqrt{\rho})} n' \leq \alpha \nu n'/2$ 
		also by the previous proposition. Note for later that 
\begin{equation}
\label{eq:volU0}
\vol_G(U_0) \leq  n |U_0| \leq (2n' / \alpha) (\alpha \nu n'/2) \leq \nu n'^2.
\end{equation}
	
	Set $U':=U\setminus U_0$ and choose $\phi$ such that $\nu\ll \phi \ll \rho'$.
	We apply Theorem~\ref{th:cheeger} to $G[U']$ as follows to construct $U_1, U_2, \ldots$. 	
	Given $U_i$, set $\overline{U_i}:=U\setminus U_i$ and $G_i:= G[\overline{U_i}]$ apply the algorithm of Theorem~\ref{th:cheeger} to $G_i$ to output some $S_i \subseteq \overline{U_i}$. By replacing $S_i$ with $U_i \setminus S_i$ if necessary, assume $|S_i| \leq |U_i \setminus S_i|$.
	If
	\begin{enumerate}[noitemsep]
		\item[] $\phi_i := \Phi_{G_i}(S_i)> \phi$ or $|U_i|\geq \frac{1}{3}|U|$
	\end{enumerate}
	then stop. Otherwise 
	set $U_{i+1}=U_i \cup S_i $ and repeat.
	In this way we obtain sets $S_0, \ldots, S_{t-1}$ and $U_0, \ldots, U_t$ in polynomial time. Note that $|U_{t-1}| < \frac{1}{3} |U|$, so 
	\begin{equation}
	\label{eq:Ut}
	 |U_t| = |U_{t-1}|+|S_{t-1}| \leq |U_{t-1}| + \frac{1}{2} (|U| -|U_{t-1}|)  \leq \frac{2}{3} |U|.
\end{equation}
	
	There are two cases to consider:
	\begin{enumerate}[noitemsep]
		\item[(a)] $|U_t| > \frac{1}{4} \rho'n'$ and
		\item[(b)] $|U_t| \leq \frac{1}{4} \rho'n'$.
	\end{enumerate}
	
	\begin{claim}\label{Claim3}
		In case (a),  $U_t$, $\overline{U_t}$ are $\rho'$-components.
	\end{claim}
	
	\begin{claim}\label{Claim4}
		In case (b), $G[U]$ is a robust $(\nu, \tau)$-expander.
	\end{claim}
	
	Since we can output $U_t,\overline{U_t}$ in polynomial time, these two claims prove Lemma~\ref{th:alg1}.
	
	\begin{proof}[\textbf{Proof} of Claim~\ref{Claim3}]
Since we are in case (a), note that $\Phi_{G_i}(S_i) \leq \phi$ for all $i=1, \ldots, t$ and so 
\begin{equation}
\label{eq:vol2}
e_G(S_i, U_i \setminus S_i) \leq \phi \vol_{G_i}(S_i) \leq \vol_G(S_i).
\end{equation}	
		Recall also that $U_t = U_0 \cup (\bigcup_{i=0}^{t-1}S_i)$.
		Using that volume is additive, i.e.\ $\vol_{G}(U_t) = \vol_G(U_0) + \sum_{i=0}^{t-1}\vol_{G}(S_i)$, we have
		\begin{align*}
			e_G(U_t,\overline{U_t}) 
= e_G(U_0, \overline{U_t}) + \sum_{i=0}^{t-1}e_G&(S_i,\overline{U_t})
\leq \vol_G(U_0) + \sum_{i=0}^{t-1}e_G(S_i, U_i \setminus S_i) \\
			&\overset{\eqref{eq:volU0}, \eqref{eq:vol2}}{\leq} \nu n'^2 + \sum_{i=0}^{t-1}\phi \; \vol_{G}(S_i) \\
			&\leq \nu n'^2 + \phi \; \vol_G(U_t) \leq  \nu n'^2 + \phi |U_t|  n.	
			\end{align*}			
Therefore
\begin{align*}	e_G(U_t,\overline{U_t}) \leq \nu n'^2 + \phi |U_t|  n	\overset{\eqref{eq:Ut}}{\leq} \nu n'^2 +  \frac{2}{3}\phi  |U|  n 
&\overset{{\rm Prop}\, \ref{Claim1}}{\leq} \nu n'^2 + \frac{\phi }{\alpha-\sqrt{\rho}}n'^2 \\
&\overset{\nu, \phi \ll \rho'}{\leq} \frac{1}{2}\rho'n'^2.
\end{align*}		
		Hence $e_G(U_t,V\setminus U_t)\leq \frac{1}{2} \rho' n'^2+\rho n^2\leq \rho' n^2$
		since $U_t \subseteq U$ and $U$ is a $\rho$-component. Similarly  $e_G(\overline{U_t},V\setminus\overline{U_t})\leq \rho' n^2$.
		Also, $|U_t|$, $|\overline{U_t}| \geq \frac{1}{4}\rho'n$ by (a) and \eqref{eq:Ut}. However, by Proposition~\ref{Claim1}, we in fact have $|U_t|$, $|\overline{U_t}| \geq (\alpha-\rho'^2)n \geq \sqrt{\rho'}n$, so $U_t$ and $\overline{U_t}$ are $\rho'$-components.
	\end{proof}
	
	\begin{proof}[\textbf{Proof} of Claim~\ref{Claim4}]
		First some observations. Since case (b) holds, $|U_t|\leq \frac{1}{4}\rho'n'\leq \frac{1}{2}\tau n'\leq \frac{1}{3}|U|$ and $\phi_t = \Phi_{G_t}(S_t)>\phi$.
		
		Also, $\delta(G_t) = \delta(G[\overline{U_t}])
		\geq \min_{x \in \overline{U_t}} d_U(x) - |U_t| \geq  \frac{1}{2}\alpha n - \frac{1}{2}\tau n' \geq \frac{1}{3} \alpha n$, where the penultimate inequality follows from our choice of $U_0$.
		By Theorem~\ref{th:cheeger}, for all $R \subseteq V(G_t)=U\setminus U_t$ we have $\Phi_{G_t}(R) \geq \Phi(G_t) \geq \phi_t^2/4 \geq \phi^2/4$, i.e.\ 
		$$ e_{G_t}(R,R')\geq \frac{\phi^2}{4} \min(\vol(R),\vol(R') )\geq \frac{1}{12}\phi^2 \alpha n \min(|R|,|R'|)$$
		where $R' = \overline{U_t}\setminus R = U\setminus (U_t \cup R)$.
	Furthermore, for $R\subseteq U$ and $\overline{R}:=U\setminus R$, we have 
		\begin{align}
		e_{G[U]}(R,\overline{R}) \geq e_{G[U]}(R\setminus U_t,\overline{R}\setminus U_t) 
		&\geq \frac{1}{12}\phi^2 \alpha n \min(|R\setminus U_t|,|\overline{R}\setminus U_t|) \notag \\
		&\geq\frac{1}{12} \phi^2 \alpha n \left( \min(|R|,|\overline{R}|) - \frac{1}{4} \rho'n' \right).
				\label{Star}
		\end{align}
		We will now show that $G[U]$ is a $(\nu,\tau)$-expander by assuming that $G[U]$ does not expand and deducing that $G[U]$ is $\rho'$-close to bipartite, contradicting the premise of the lemma.
		
		Suppose there exists $ S\subseteq U$ with $\tau n' \leq |S| \leq (1-\tau)n'$ such that
		$N=\RN_{\nu,G[U]}(S)\text{ satisfies } |N|<|S|+\nu n. $
		Since $\tau n' \leq |S| \leq (1-\tau)n'$, we have $\frac{1}{4}\rho'n'\leq \frac{1}{2}\tau n' \leq \frac{1}{2}\min(|S|,|\overline{S}|)$ so by \eqref{Star}, we have
		\begin{equation}\label{StarStar}
			e_{G[U]}(S,\overline{S})\geq \frac{1}{24}\phi^2 \alpha n\min(|S|,|\overline{S}|).
		\end{equation}
		\begin{claim}
			We may assume $\frac{1}{4}\alpha n\leq |S| \leq |U|-\frac{1}{4}\alpha n$.
		\end{claim}
		\begin{proof}[Proof of claim]
			If $|S|<\frac{1}{4}\alpha n$ then $e_G(S,\overline{S})\geq |S|(\alpha n - |S|)-\rho n^2$ and $e_G(S,\overline{S})\leq |N||S|+|U\setminus N| \nu n \leq |N||S| + \nu n^2$, so
combining these inequalities and rearranging, we obtain	
			\begin{align*}	
				|N| \geq \alpha n - &|S| - (\rho + \nu )\frac{n^2}{|S|}\geq \alpha n - |S| - (\rho + \nu )\frac{n^2}{\tau n'} \\
				&\overset{{\rm Prop}\,\ref{Claim1}}{\geq} \alpha n - \frac{1}{4}\alpha n - (\rho + \nu )\frac{n'}{\tau (\alpha -\sqrt{\rho})^2}  \geq \frac{1}{2}\alpha n' \geq |S|+\nu n,
			\end{align*}
			contradicting our choice of $S$.
			
			Similarly if $|S| > |U| - \frac{1}{4}\alpha n$ recall 
			that by Proposition~\ref{Claim1} that all but the $\gamma n'$ 
			vertices in $U_0$ have degree at least $\frac{1}{2}\alpha n$ in $U$ and so for all $x \in  U \setminus U_0$, we have
			$$d_S(x) \geq \frac{1}{2}\alpha n - |U\setminus S| \geq \frac{1}{4}\alpha n \geq \nu n.$$ 
			Hence $N\supseteq U\setminus U_0$ and so $|N|\geq |U| - |U_0| \geq (1- \nu)n' \geq |S|+\nu n'$, a contradiction. This proves the claim.
		\end{proof}
		
		\begin{figure}
			\centering
			\includegraphics[width=0.3\textwidth]{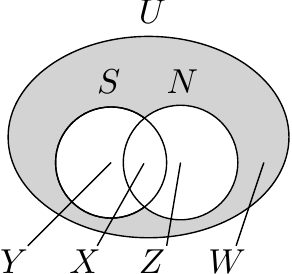}
			\caption{Overview of subsets mentioned in the coming section.}
			\label{fig}
		\end{figure}
		We continue with the proof of Claim~\ref{Claim4}.
		We define $Y = S\setminus N$, $X = S \cap N$, $Z = N\setminus S$, $W = U\setminus (S\cup N)$; see Figure~\ref{fig}.
		Since each vertex in $Y$ has at most $ \nu n$ neighbours in $S$ and since $G$ is $D$-regular and $U$ is a $\rho$-component we have $e_G(Y,\overline{S}) \geq D |Y| - \rho n^2 - \nu n^2$. Using this, we obtain
		\begin{align}
		\label{(eYZ)}
 e_G(Y,Z) = e_G(Y,\overline{S})-e_G(Y,W)			
			&\geq D |Y|- \rho n^2 - \nu n^2 - |W|\nu n \notag \\ &\geq D |Y| - 3\nu n^2.
		\end{align}
On the other hand $e_G(Z,Y) \leq D|Z|$, which together with \eqref{(eYZ)} implies after rearranging that $|Z| \geq |Y| - \frac{3\nu}{\alpha}n$. Also $|Z| \leq |Y|+\nu n$;  otherwise $S$ does not violate $(\nu,\tau)$-expansion. 
Hence we have shown
		\begin{equation}
			\label{(1)}
			|Y| - \frac{3\nu}{\alpha}n\leq |Z| \leq  |Y|+\nu n.
		\end{equation}
		Considering $W$ (and taking $\overline{W}:= U \setminus W$), we see
		$$e_G(W,\overline{W})= e_G(W,S)+e_G(Z,W) \leq e_G(W,S)+(D|Z| - e_G(Z,Y))$$
		\begin{equation}
		\label{(eWW)}
		\hspace*{0.71em} 
		\overset{\eqref{(1)},\eqref{(eYZ)}}{\leq} \nu n^2 + D(|Y|+\nu n)- (D|Y| - 3\nu n^2)\leq 5 \nu n^2,
		\end{equation}
		as well as
		\begin{align*}
			\frac{1}{12}\phi^2\alpha n \min(|W|,|\overline{W}|)&-\frac{1}{48}\phi^2\alpha\rho'nn'\overset{\eqref{Star}}{\leq} e_G(W,\overline{W})\overset{\eqref{(eWW)}}{\leq}5\nu 
			n^2. 
		\end{align*}
		Since $|\overline{W}|\geq |S|\geq \tau n'>2\rho'n'$, we must have 
		\begin{equation}
			\label{(sizeW)}
			|W| \leq \frac{60\nu n}{\phi^2 \alpha}+ \frac{1}{4}\rho'n'\leq \frac{1}{2}\rho'n'.
		\end{equation}
		Now consider $Y\cup Z$. We have 
		\begin{align}
		e_G(Y\cup Z, \overline{Y\cup Z}&) \leq D |Y\cup Z| - 2e_G(Y,Z) \notag \\ 
		&\overset{\eqref{(1)},\eqref{(eYZ)}}{\leq}D (2|Y|+\nu n) - 2 (D|Y| - 3\nu n^2)	
		\leq 7\nu n^2.
		\label{(eYZYZ)}
		\end{align}
		Combining this with an application of \eqref{Star}
		$$\frac{1}{12}\phi^2 \alpha n(\min(|Y\cup Z|,|\overline{Y\cup Z}|)- \frac{1}{4}\rho'n')\overset{\eqref{Star}}{\leq}e_G(Y\cup Z,\overline{Y\cup Z})\overset{\eqref{(eYZYZ)}}{\leq} 7\nu n^2,$$
and hence
		$$ \min(|Y\cup Z|,|\overline{Y\cup Z}|)\leq 84\frac{\nu n}{\phi^2\alpha}+ \frac{1}{4}\rho'n' \leq \frac{1}{2}\rho'n'.$$
		If $|Y\cup Z|\leq \frac{1}{2}\rho'n'$, then
		\begin{align*}
		|S| &= |U|-|W|-|Z| \geq |U|-|W|-|Y\cup Z| \\ &\overset{\eqref{(sizeW)}}{\geq} n'-\frac{1}{2}\rho'n'-\frac{1}{2}\rho'n'  \geq (1-\tau)n',	
		\end{align*}
		a contradiction. So we have 
		\begin{equation}
			\label{(sizeYZ)}
			|\overline{Y\cup Z}| \leq \frac{1}{2}\rho'n'.
		\end{equation}
		Finally we show that $Y$,$\overline{Y}$ gives a partition that shows $G[U]$ is $\rho'$-close to bipartite, giving a contradiction. Note that $|\overline{Y}| = |Z| + |\overline{Y\cup Z}|$, so
	
		\begin{align*}
		|Y| - \frac{3\nu}{\alpha}n		 \overset{\eqref{(1)}}{\leq}|Z|\leq |\overline{Y}|= |Z| +
		|\overline{Y\cup Z}|\overset{\eqref{(1)}, \eqref{(sizeYZ)}}{\leq}& |Y|+\nu n + \frac{1}{2}\rho'n' 	\\
		&\leq |Y| + \frac{3}{4}\rho'n'.
		\end{align*}
		Therefore,
		\begin{equation}
		\label{(sizeYY)}
			||Y|-|\overline{Y}||\leq \frac{3}{4}\rho'n'.			
		\end{equation}

		\noindent If $\rho'$ is small enough, e.g. $\rho' \leq \frac{1}{10}$, this also gives us $|Y|$, $|\overline{Y}|\geq \sqrt{\rho'}n'.$ Also
		
		$$ e_G(Y,V\setminus\overline{Y})+e_G(\overline{Y},V\setminus Y)\leq D |Y\cup \overline{Y}|-2e_G(Y,\overline{Y})$$
		$$\leq D |U| - 2e_G(Y,Z) \overset{\eqref{(eYZ)}}{\leq} D n' -2(D|Y|-3\nu n^2)$$
		$$\overset{\eqref{(sizeYY)}}{\leq} Dn' - D |Y| - D \left( |\overline{Y}|-\frac{3}{4}\rho'n' \right) +6\nu n^2 \leq \frac{4}{5} D\rho'n' \leq \rho'n'^2. $$
		So  $Y$,$\overline{Y}$ is a partition of $U$ showing $G[U]$ is $\rho'$-close to bipartite, a contradiction,
		completing the proof of the claim and the lemma.
	\end{proof}
\end{proof}

\begin{proof}[Proof of Lemma~\ref{th:alg2}]
	
		The idea is to repeatedly apply the algorithm in Theorem~\ref{th:trevisan} and iteratively remove vertices that are assigned to bipartite parts until the remaining induced graph is either small or far from bipartite.
	
	We choose $\beta$ such that $\rho \ll \beta \ll \rho'$.
	Set $U_0=\emptyset$, and given $U_i$, let $G_i=G[U\setminus U_i]$. Let $y$ be obtained from running the algorithm in Theorem~\ref{th:trevisan} on $G_i$. We set $U_{i+1}=U_i\cup A_i \cup B_i$, where $A_i:=\{v\mid y_v=1\}$ and $B_i:=\{v\mid y_v=-1\}$ and we set $\beta_i = \beta(y)$.  Note that $G_{i+1}\subset G_i$. We continue until either
	\begin{enumerate}[noitemsep]
		\item[(a)] $|G_i|\leq \rho'n$ or
		\item[(b)] $\beta_i \geq \beta$.
	\end{enumerate}
	Let $t$ be the first index where (a) or (b) occurs.	
	\begin{claim}
		\label{cl:closebip}
		If $|G_t|\leq \rho'n$, then $G[U]$ is $\rho'$-close to bipartite.
	\end{claim}
	\begin{claim}
		\label{cl:notclosebip}
		If $\beta_t > \beta$ and $|G_t|\geq \rho'n$, then $G[U]$ is not $\rho$-close to bipartite.
	\end{claim}

Note that these two claims together prove the lemma since we can compute the $\beta_i$ and the $G_i$ in polynomial time (and for the first claim, the proof will show how to compute the corresponding partition). 	
	
	\begin{proof}[\textbf{Proof} of Claim~\ref{cl:closebip}]
		Let $R = U \setminus U_t$, i.e.\ the set of vertices that are not part of some $A_j$ or $B_j$ for $j \leq t$. Note that $|R| \leq \rho'n$.
For each $j \leq t$, using the definition of $A_j, B_j$ and \eqref{eq:beta}, we have
$$E_j := 2e_{G_j}(A_j)+2e_{G_j}(B_j)+e_{G_j}(A_j\cup B_j,U\setminus U_{j+1})  \leq \beta \text{ }\vol_{G_j}(A_j\cup B_j).$$
		
		First we note that for each $j \leq t$, we have
\begin{align}
e_G(U_j, U \setminus U_j) 
\leq \sum_{i=0}^{j-1}e_G(A_i \cup B_i, U \setminus U_{i+1}) 
&\leq \beta \sum_{i=0}^{j-1} \vol_{G_i}(A_i \cup B_i)\notag \\ 
&\leq \beta \vol_G(U_j) \leq \frac{1}{10} \rho'Dn, \label{eq:Uj}
\end{align}
where the final inequality follows by our choice of $\beta \ll \rho'$ and $\vol_G(U_j) \leq D n$.
In particular, for each $j<t$, we have 
$$e_G(A_j, U_j) \leq e_G(U_j, U \setminus U_j)  \leq \frac{1}{10} \rho' Dn.$$
	Next, we claim that for each $j$, $||A_j|-|B_j||\leq \rho'n$. Assume for a contradiction that $|A_j|-|B_j|\geq \rho'n$ for some $j$. First we note that 
		\begin{align*}
			e_{G_j}(A_j,\overline{B_j})
			&\geq (|A_j| - |B_j|)D - e_G(U,\overline{U})-e_G(A_j,U_{j}) \\
			&\geq \rho' D n - \rho n^2 -  \frac{1}{10} \rho' Dn 
			\geq \frac{1}{2}\rho'D n, 
\end{align*}
using $\rho \ll \rho'$ for the last inequality.
On the other hand we have $e_{G_j}(A_j,\overline{B_j}) \leq e_G(A_j, U_j) \leq \frac{1}{10}\rho' D n$			
		a contradiction.
		
%
%
%
%

By the preceding claim, we can form a partition $A,B$ of $U$ such that (i) for each $j < t$, either $A_j \subseteq A \wedge B_j \subseteq B$ or $A_j \subseteq B \wedge B_j \subseteq A$ and (ii) $||A| - |B|| \leq \rho'n$. Indeed we can start with an arbitrary partition satisfying (i) and then iteratively swap suitable $A_j$ and $B_j$ if this reduces the value of $||A| - |B||$.(Note that $A$ and $B$ also contain vertices of $R$ (i.e.\ vertices not belonging to any $A_j$ or $B_j$) that can be freely moved to reduce  $||A| - |B||$). It is easy to see $A,B$ can be computed in polynomial time and we shall see below that this partition demonstrates that $G[U]$ is $\rho'$-close to bipartite.	
		
		To see this, we count edges not in $E_G(A,B)$. We have
		$$e_G(A)+e_G(B) +e_G(A\cup B, \overline{U}) \leq \sum_{j=0}^{i-1}E_j +\vol_{G[R]}(R)+e_G(U, \overline{U})$$
		$$\leq \underbrace{\beta}_{\ll \rho'}\text{ }\underbrace{\vol_{G[U]}(U\setminus R)}_{\leq n^2}+ (\rho'n)^2 + \rho n^2  \leq \rho'n^2.$$
	\end{proof}
	
	\begin{proof}[\textbf{Proof} of Claim~\ref{cl:notclosebip}]
		
		Define 
		\begin{align*}
			\beta'(G):=&\min_{y\in \{-1,1\}^{V(G)}}\frac{\sum_{uv\in E(G)} |y_u+y_v|}{\sum_{v\in V(G)} d_G(v)  |y_v|} \geq \beta(G),\\
			\overline{\beta}(G):=&\min_{A,B \text{ bipartition of }G}e_G(A)+e_G(B).
		\end{align*}
		
		Then we have $\overline{\beta}(G[U])\geq \overline{\beta}(G_t)$ and recalling that $V(G_t) = U \setminus U_t$, we have
		\begin{equation}
		\label{eq:4.1}
		2\frac{\overline{\beta}(G_t)}{\vol_{G_t}(U \setminus U_t)}= \beta'(G_t)\geq \beta(G_t)
		\geq \frac{\beta_t^2}{4}\geq \frac{\beta^2}{4},
		\end{equation}
		where we use the definition of $\beta_i$ and Theorem~\ref{th:trevisan}.
Then we have
		\begin{align*}
			\vol_{G_t}(U \setminus U_t) &\geq D|U \setminus U_t|-\rho n^2 -e_G(U_t,U\setminus U_t) \\
			&\geq \rho' D n -\rho n^2 - \frac{1}{10} \rho' Dn \geq \frac{1}{2} \rho' Dn,
		\end{align*}
		where we have used that $U$ is a $\rho$-component, \eqref{eq:Uj}, and $\rho \ll \rho'$.
		Combining with \eqref{eq:4.1} we see
		$$\overline{\beta}(G)\geq \frac{\beta^2}{8}\text{ }\vol_{G_t}(U \setminus U_t) \geq \frac{\beta^2}{16}  \rho' Dn>\rho n^2.$$	
		\end{proof}
		This completes the proof of the lemma.
\end{proof}

\begin{proof}[Proof of Lemma~\ref{th:alg3}]
	Fix $\phi$ such that $\nu \ll \phi \ll \rho'$.
	As in Lemma~\ref{th:alg1}, we use algorithm in Theorem~\ref{th:cheeger} to iteratively find poorly connected subgraphs of $G[U]$ and remove them. 
	
In polynomial time, we can find $S_0, \ldots, S_{t-1}$, $U_0, \ldots, U_t$, and $\phi_1, \ldots, \phi_t$, which are defined and found in exactly the same way as in the proof of Lemma~\ref{th:alg1}, so again, we have $\phi_t > \phi$ or $|U_t|\geq \frac{1}{3}|U|$. There are two cases:
	\begin{enumerate}[noitemsep]
		\item[(a)] $|U_t|> \frac{1}{4}\rho'n'$ and
		\item[(b)] $|U_t|\leq \frac{1}{4}\rho'n'$.
	\end{enumerate}
	\begin{claim}
		In case (a), $U_t$, $\overline{U_t}$ are $\rho'$-components.
	\end{claim}
	Noting that $G[U]$ is a $\rho$-component, the proof of Claim~\ref{Claim3} holds here as well.
	\begin{claim}
		In case (b), $G[U]$ is a robust bipartite $(\nu,\tau)$-expander with bipartition $A,B$.
	\end{claim}

Once again, the two claims together prove the lemma since we can compute $U_t, \overline{U}_t$ (which give the partition $U_1, U_2$ in the statement of the lemma) in polynomial time.	
	
	\begin{proof}
		As in \eqref{Star} in the proof of Claim~\ref{Claim4}, for $S\subseteq U$ and $\overline{S}=U\setminus S$ we have 
		\begin{equation}
			\label{thm2:eq1}
			e_{G[U]}(S,\overline{S})
			\geq \frac{1}{12} \phi^2 \alpha n \left( \min(|S|,|\overline{S}|) - \frac{1}{4} \rho'n' \right)
			\end{equation}
		We will show that $G[U]$ is a bipartite robust expander by assuming the existence of a non-expanding set and finding a contradiction. 
		
		Suppose $A^*\subseteq A$ with $\tau|A|\leq |A^*|\leq (1-\tau)|A|$, let $B^*:= \RN_{G[U]}(A^*)\cap B$ and assume $|B^*|< |A^*|+\nu n$. 
		Define $\hat{A}:=A\setminus A^*$ and $\hat{B}:= B \setminus B^*$.
		We will give an upper bound on $e_G(A^*\cup B^*,\hat{A}\cup \hat{B})$ that contradicts \eqref{thm2:eq1}. Indeed, we have (suppressing the subscript $G$)
\begin{align*}
e(A^*\cup B^*,\hat{A}\cup \hat{B}) 
&\leq e(A^*,\hat{A}) + e(B^*,\hat{B}) + e(A^*,\hat{B}) + e(B^*, \hat{A}) \\
&\leq \rho n^2 + \nu n^2 + e(B^*, \hat{A}),  
\end{align*}
where we used that $e(A^*,\hat{A}) + e(B^*,\hat{B}) \leq \rho n^2$ (since $G$ is $\rho$-close to bipartite) and $e(A^*,\hat{B}) < \nu n^2$ (since every vertex in $\hat{B}$ has at most $\nu n$ neighbours in $A^*$). In order to bound $e(B^*, \hat{A})$, we have
\begin{align*}
e(B^* \hat{A}) 
&\leq |B^*|D - e(B^*, A^*) \\
&\leq (|A^*| + \nu n)D - [|A^*|D - e(A^*, \hat{A}) - e(A^* \hat{B}) - e(A^*, \overline{U})] \\
&\leq \nu n|D| + \rho n^2 + \nu n^2
\leq \rho n^2  + 2 \nu n^2,
\end{align*}
 where we used that $e(A^*, \hat{B}) \leq \nu n^2$ (as above) and $e(A^*, \hat{A}) + e(A^*, \overline{U}) \leq \rho n^2$ (since $U$ is $\rho$-close to bipartite).
Combining, we obtain
		\begin{equation}
		\label{eq:eAB}
		e_G(A^*\cup B^*,\hat{A}\cup\hat{B})
		\leq 2\rho n^2+ 3\nu n^2 \leq 5 \nu n^2.
		\end{equation}
		 However, as $\min(|A^*\cup B^*|,|\hat{A}\cup\hat{B}|)\geq \tau |A| \geq \tau \frac{1}{3}|U|$ (using Remark~\ref{rem1}), with \eqref{thm2:eq1} we have
		$$ e_{G}(A^*\cup B^*,\hat{A}\cup\hat{B})\geq \frac{1}{12}\phi^2 \alpha n \left( \frac{1}{3}\tau |U|-\frac{1}{4}\rho'|U| \right)> 5\nu n^2,$$
using $|U| \geq \frac{1}{2}\alpha n$ by Proposition~\ref{Claim1} and our choice of parameters, which contradicts \eqref{eq:eAB}.
	\end{proof}
	This completes the proof of the lemma.
\end{proof}

\begin{proof}[Proof of Lemma~\ref{th:alg6}]
	Fix $\rho_1, \rho_2, \nu_2 $ such that $\rho \ll \nu \ll \rho_1 \ll \rho_2 \ll \nu_2 \ll \rho'$.
	We run Algorithm 2 on $U$ with $(\rho_1, \rho_2)$ playing the roles of $(\rho, \rho')$. The algorithm determines either that 
	\begin{itemize}[noitemsep]
	\item $G[U]$ is not $\rho_1$-close to bipartite, or
	\item $G[U]$ is $\rho_2$-close to bipartite (and outputs a bipartition $A,B$ of $U$ that demonstrates this).
	\end{itemize}
		
	In the first case, we apply Algorithm 1 with $(\rho, \nu, \rho_1)$ playing the roles of $(\rho, \nu, \rho')$ and the algorithm either concludes that $G[U]$ is a robust $(\nu, \tau)$-expander, or it outputs a partition $U_1, U_2$ of $U$ such that $U_1$ and $U_2$ are $\rho_1$-components and hence are also $\rho'$-components.
	
	In the second case, we apply Algorithm 3 with $(\rho_2, \nu_2, \rho')$ playing the roles of $(\rho, \nu, \rho')$ and the algorithm either concludes that $G[U]$ is a bipartite robust $(\nu_2, \tau)$-expander and hence also a bipartite robust $(\nu, \tau)$-expander (and it outputs a bipartition $A,B$ of $U$ to demonstrate this) or it outputs a partition $U_1, U_2$ of $U$ such that $U_1$ and $U_2$ are $\rho'$-components.
%
\end{proof}

\subsection{Recognising robust expanders}
\label{sec:digression}
In this subsection, we make a small digression to partially address a question of K{\"u}hn and Osthus from \cite{Robb2}; the result of this subsection will not be needed in the remainder of the paper. Using the Szemer{\'e}di Regularity Lemma, K{\"u}hn and Othus \cite{Robb2} give a polynomial time algorithm for deciding whether a graph\footnote{In fact their algorithm works more generally for digraphs} is a robust $(\nu, \tau)$-expander or whether it is not a $(\nu', \tau)$-expander (provided $\nu \ll \nu'$, which is what one is interested in all applications). They asked whether the use of the Szemer{\'e}di Regularity Lemma can be avoided, and we answer this affirmatively for regular graphs.

\begin{corollary}
	For each fixed choice of parameters $0 \leq \nu \ll  \nu' \ll \tau \ll \alpha < 1$ there exists a polynomial-time algorithm that does the following. Given  a $D$-regular $n$-vertex graph $G=(V,E)$, where $D \geq \alpha n$, the algorithm determines that either
	
	\begin{enumerate}[noitemsep]
		\item[(i)] $G$ is a robust $(\nu,\tau)$-expander, or
		\item[(ii)] $G$ is not robust $(\nu',\tau)$-expander,
	\end{enumerate}
and in case (ii) the algorithm finds a set $S \subseteq V$ such that $\tau n \leq |S| \leq (1- \tau)n$ and $|\RN_{\nu', G}(S)| \leq |S| + \nu' n$.
\end{corollary}
\begin{proof}
The proof is a variation of the previous lemma. First choose parameters $1/n_0 \ll \rho \ll \nu \ll \rho_1 \ll \rho_2 \ll \nu' \ll \tau \ll \alpha \ll 1$. If $n \leq n_0$ then we check whether (i) or (ii) holds by exhaustive search in constant time. 

If $n \geq n_0$, we apply Algorithm~2 to $G$ with $(\rho_1,  \rho_2, V)$ playing the roles of $(\rho, \rho', U)$ (and thinking of $G = G[V]$ as a $\rho_1$-component of $G$). The algorithm determines that either
\begin{enumerate}[noitemsep]
\item[(a)] $G$ is $\rho_2$-close to bipartite (and gives a partition $A,B$ of $V$ showing this), or
\item[(b)] $G$ is not $\rho_1$-close to bipartite.
\end{enumerate}
In case (b) we apply Algorithm~1 with $(\rho, \nu, \rho_1, V)$ playing the roles of $(\rho, \nu, \rho', U)$ (and thinking of $G = G[V]$ as a $\rho$-component of $G$), and the algorithm determines that either
\begin{enumerate}[noitemsep]
\item[(bi)] $G=G[V]$ is a robust $(\nu, \tau)$-expander;
\item[(bii)] $U=V$ has a partition $U_1, U_2$ such that $U_1$, $U_2$ are $\rho_1$-components.
\end{enumerate}
In case (bi), we are done. In case (a) and (bii), we show $G$ is not a robust $(\nu', \tau)$-expander. Indeed, in case (a), assume that $|A| \leq |B|$. We have $|A|,|B| \geq \frac{1}{2} \alpha n \geq 2 \tau n$ by Remark~\ref{rem1}, so $\tau n \leq |B| \leq (1-\tau)n$. We cannot have that $|\RN_{\nu',G}(B)| \geq |B| + \nu' n$, for otherwise $|\RN_{\nu',G}(B) \cap B| \geq \nu' n$ and therefore $e_G(B,\overline{A}) = e_G(B) \geq \frac{1}{2}\nu'^2 n^2 > \rho_2 n^2$, contradicting that $G$ is $\rho_2$-close to bipartite.  So $G$ is not a robust $(\nu', \tau)$-expander in this case and the algorithm outputs $S = B$.

Similarly in case (bii) we know that $|U_1|, |U_2| \geq \frac{1}{2}\alpha n \geq 2 \tau n$ by Proposition~\ref{Claim1} and so $\tau n \leq |U_1| \leq (1- \tau)n$. Also, we cannot have that $|\RN_{\nu',G}(U_1)| \geq |U_1| + \nu' n$, for otherwise $|\RN_{\nu',G}(U_1) \cap U_2| \geq \nu' n$ and therefore $e_G(U_1,U_2) \geq \nu'^2 n^2 > \rho_1 n^2$, contradicting that $U_1$ is a $\rho_1$-component. So $G$ is not a robust $(\nu', \tau)$-expander in this case and the algorithm outputs $S = U_1$.
\end{proof}

\subsection{Assembling the robust partition}

We begin with several basic facts from \cite{kuhn2014robust}. The first three are basic facts about (bipartite) robust expanders, which are taken from \cite{kuhn2014robust} unchanged and their proofs are included for completeness.
 
\begin{lemma}
	\label{lem:stillrobustexp}
	Let $0<\nu\ll\tau<1$. Suppose that $G$ is a graph and $U$, $U' \subseteq V(G)$ are such that $G[U]$ is a robust $(\nu,\tau)$-expander and $|U \triangle U'|\leq \nu |U|/2$. Then $G[U']$ is a robust $(\nu/2,2\tau)$-expander
\end{lemma}
\begin{proof}
	The statement immediately follows by considering a set $S\subseteq U'$ with $2\tau|U'|\leq |S|\leq (1-2\tau)|U'|$ and considering its robust neighbourhood. As $\tau|U|\leq |S\cap U|\leq(1-\tau)|U|$, we have $|\RN_{\nu,U}(S\cap U)|\geq |S\cap U|+\nu|U|\geq |S|-|U\setminus U'|+\nu|U|$. With $|\RN_{\nu,U}(S\cap U)\cap U'|\geq |\RN_{\nu,U}(S\cap U)|-|U'\setminus U|$ it follows that $|\RN_{\nu/2,U'}(S)|\geq |S|+\nu/2|U'|$.
\end{proof}

\begin{lemma}
	\label{lem:stillRobustExpComp}
	Let $0<\rho\leq\gamma\ll\nu\ll\tau<1$. Suppose that $G$ is a graph and  $T\subseteq U \subseteq V(G)$ are such that $G[U]$ is a robust $(\rho,\nu,\tau)$-expander component, $|T|\leq \rho n$. Then $G[U\setminus T]$ is a robust $(3\gamma,\nu/2,2\tau)$-expander component.
\end{lemma}
\begin{proof}
	We have that $|U\setminus T| = |U|-|T| \geq \alpha n - \sqrt{\rho}n - \rho n \geq \sqrt{3\gamma} n$, where we use Proposition~\ref{Claim1}(i) for the first inequality. Next we see $e_G(U\setminus T, \overline{U\setminus T})\leq e_G(U,\overline{U})+D\rho n\leq\rho n^2+  \rho n^2 \leq 3\gamma n^2$, showing that $G[U\setminus T]$ is a $3\gamma$-component.
	Finally, $G[U\setminus T]$ is a $(\nu/2,2\tau)$-expander by Lemma~\ref{lem:stillrobustexp}.
\end{proof}

\begin{lemma}
	\label{lem:stillbiprobustexp}
	Let $0<1/n\ll\rho\leq\gamma\ll\nu\ll\tau\ll\alpha<1$ and suppose that $G$ is a $D$-regular graph on $n$ vertices where $D\geq \alpha n$.
	\vspace{0.2cm}
	
	(i) Suppose that $G[A\cup B]$ is a bipartite $(\rho,\nu,\tau)$-robust expander component of $G$ with bipartition $A,B$. Let $A',B'\subseteq V(G)$ be such that $|A\triangle A'| + |B\triangle B'|\leq \gamma n$. Then $G[A'\cup B']$ is a bipartite $(3\gamma, \nu/2, 2\tau)$-robust expander component of $G$ with bipartition $A',B'$.
	
	(ii) Suppose that $G[U]$ is a bipartite $(\rho,\nu,\tau)$-robust expander component of $G$. Let $U'\subseteq V(G)$ be such that $|U\triangle U'|\leq \gamma n$. Then $G[U']$ is a bipartite $(3\gamma, \nu/2, 2\tau)$-robust expander component of $G$.
\end{lemma}
\begin{proof}
	We start with (i). To see that $G[A'\cup B']$ is $3\gamma$-close to bipartite, we see that $|A'|,|B'|\geq D-2\sqrt{\rho}\geq \sqrt{3\gamma}n$ by Remark~\ref{rem1}. We have that $||A'|-|B'||\leq ||A|-|B||+\gamma n\leq 3\gamma n$ and $e(A',\overline{B'})+e(B',\overline{A'})\leq e(A,\overline{B})+e(B,\overline{A})+2(|A'\triangle A|+|B'\triangle B|)n \leq 3\gamma n$. $G[A'\cup B']$ is a bipartite $(\nu/2,2\tau)$-robust expander by a straightforward calculation as in the proof of Lemma~\ref{lem:stillrobustexp}.
	It is easy to see that part (ii) follows from (i).
\end{proof}

The non-algorithmic versions of the next two lemmas can be found in \cite{kuhn2014robust}; we use a simple greedy procedure to make them algorithmic. These lemmas will be used later to ensure conditions (D4), (D5), and (D7) when constructing our robust partition.

\begin{lemma}
	\label{lem:modifycomp}
	Let $m,n,D\in \mathbb{N}$ and $0<1/n_0 \ll\rho\ll\alpha, 1/m\leq 1$. Let $G$ be a $D$-regular graph on $n$ vertices where $n \geq n_0$ and $D\geq \alpha n$. Suppose that $\mathcal{U}:=\{U_1,\dots,U_m\}$ is a partition of $V(G)$ such that $U_i$ is a $\rho$-component for each $1\leq i\leq m$. Then $G$ has a vertex partition $\mathcal{V}:=\{V_1,\dots,V_m\}$ such that 
	\begin{enumerate}[noitemsep]
		\item[(i)] $|U_i\triangle V_i| \leq \rho^{1/3}n$; 
		\item[(ii)] $V_i$ is a $\rho^{1/3}$-component for each $1\leq i\leq m$;
		\item[(iii)] if $x\in V_i$, then $d_{V_i}(x) \geq d_{V_j}(x)$ for all $1\leq i,j \leq m$. In particular, $d_V(x)\geq D/m$ for all $x\in V$ and all $V\in  \mathcal{V}$;
		\item[(iv)] for all but at most $\rho^{1/3}n$ vertices $x\in V_i$ we have $d_{V_i}(x)\geq D-2\sqrt{\rho}n$.
	\end{enumerate}
	Furthermore, (for fixed $n_0, \rho, \alpha, m$ satisfying the hierarchy above)  there is an algorithm that finds such a vertex partition $\mathcal{V}$ in time polynomial in $n$.
\end{lemma}

\begin{proof}
	For each $1\leq i\leq m$, let $X_i$ be the collection of vertices $y\in U_i$ with $d_{\overline{U_i}}(x)\geq \sqrt{\rho}n$. Since $U_i$ is a $\rho$-component, we have $|X_i|\leq \sqrt{\rho}n$  (otherwise $e(U_i,\overline{U_i})\geq \rho n^2$). Let $W_i:=U_i\setminus X_i$. Then each $x\in W_i$ satisfies
	\begin{equation}
		\label{eq:modcomp}
		d_{W_i}(x)=D-d_{\overline{U_i}\cup X_i}(x)\geq D-\sqrt{\rho}n - |X_i|\geq D-2\sqrt{\rho}n.
	\end{equation}
	We now redistribute the vertices of $X:=\cup_{1\leq i\leq m}X_i$ as follows:
	Iteratively move any $x\in X\cap U_i$ to $U_j$ where $j=\arg\max_i d_{U_i}(x)$ until this is no longer possible. 
	This process terminates, as the number of edges crossing the partition is reduced with each step. It is easy to see that this redistribution can be done in time polynomial in $n$. Call the resulting partition $\mathcal{V}:=\{V_1,\dots,V_m\}$, (so $V_i=W_i\cup X_i'$ for some $X_i'\subseteq X$ and $X=\sqcup X'_i$.)
	
	
	We show that $\mathcal{V}$ fulfils (i)-(iv). It is easy to see that (iii) holds by our choice of $\mathcal{V}$ for all $x\in X$. For $x\in W_i$, \eqref{eq:modcomp} implies $d_{V_i}(x)\geq d_{W_i}(x)\geq D-2\sqrt{\rho}n\geq D/2$, so (iii) holds. Next, since each step of our procedure reduces the number of edges crossing the partition, we have
	$$\sum_{1\leq i\leq m}e(V_i,\overline{V_i})\leq \sum_{1\leq i \leq m} e(U_i,\overline{U_i})\leq \rho mn^2 \leq \rho^{1/3}n^2 $$
	and therefore each $V_i$ is a $\rho^{1/3}$-component, so (ii) holds. We have $|U_i\triangle V_i| \leq |X|\leq m \sqrt{\rho}n\leq \rho^{1/3}n$ for all $i$, so (i) holds as well. To see (iv), note that for all $x\in W_i$ we have $d_{V_i}(x)\geq D-2\sqrt{\rho}n$ by \eqref{eq:modcomp} and $|V(G) \setminus \cup_{i=1}^m W_i| =  |X|\leq \rho^{1/3}n$.
\end{proof}

\begin{lemma}
	\label{lem:modifybip}
	Let $0<1/n_0 \ll\rho\ll\nu\ll\tau\ll\alpha < 1$ and let $G$ be a $D$-regular graph on $n$ vertices where $n \geq n_0$ and $D\geq \alpha n$. Suppose that $U$ is a bipartite $(\rho,\nu,\tau)$-robust expander component of $G$ with bipartition $A$, $B$. Then there exists a bipartition $A'$, $B'$ of $U$ such that
	\begin{enumerate}[noitemsep]
		\item[(i)] $U$ is a bipartite $(3\sqrt{\rho},\nu/2,2\tau)$-robust expander component with  partition $A'$ ,$B'$;
		\item[(ii)] $d_{B'}(u)\geq d_{A'}(u)$ for all $u\in A'$, and $d_{A'}(v)\geq d_{B'}(v)$ for all $v\in B'$.
	\end{enumerate}
	Furthermore, (for fixed $n_0, \rho, \nu, \tau, \alpha$ satisfying the hierarchy above) there is an algorithm that finds such a partition in time polynomial in $n$.
\end{lemma}
\begin{proof}
	This proof is similar to that of Lemma~\ref{lem:modifycomp}. Let $A_0:=\{x\in A\mid d_{\overline{B}}(x)\geq 2\sqrt{\rho}n\}$ and define $B_0$ similarly. The fact that $U$ is a $\rho$-component implies that
	\begin{align*}
		\rho n^2\geq& e(A,\overline{B})+e(B,\overline{A})\geq \frac{1}{2} \left(\sum_{x\in A}d_{\overline{B}}(x)+ \sum_{x\in B}d_{\overline{A}}(x)\right) \\
		\geq& \frac{1}{2} \left(\sum_{x\in A_0}d_{\overline{B}}(x)+ \sum_{x\in B_0}d_{\overline{A}}(x)\right)\geq (|A_0|+|B_0|)\sqrt{\rho}n
	\end{align*}
	and therefore  $|A_0|+|B_0|\leq \sqrt{\rho}n$. Define $\hat{A}:=A\setminus A_0$ and $\hat{B}:= B\setminus B_0$. For all $x\in \hat{A}$ we have
	$d_{\hat{B}}(x)\geq D-d_{\overline{B}}(x)-|B_0|\geq D-3\sqrt{\rho}n$ and an analogous statement holds for $x\in \hat{B}$.
	We iteratively move vertices between $A_0$ and $B_0$ as follows: for $x\in A_0$ if $d_A(x)>d_B(x)$ then move $x$ from $A_0$ to $B_0$ and for $y\in B_0$ if $d_B(y)>d_A(y)$ then move $y$ from $B_0$ to $A_0$ (and update $A,B, A_0, B_0$ accordingly). Continue this until it is no longer possible.
	This process terminates, as the number of edges not crossing the partition is reduced at each step. 
	It is easy to see that this redistribution can be done in time polynomial in $n$. Call the resulting parts $A'$, $B'$.
	We show that $A'$, $B'$ fulfil (i) and (ii).
	
	The choice of $A'$, $B'$ implies that all $x\in (A_0\cup B_0)$ fulfil (ii). For $x\in \hat{A}$ we have $D_{B'}(x)\geq d_{\hat{B}}(x)\geq D-3\sqrt{\rho}n\geq d_U(x)/2$. A similar statement holds for all $x\in \hat{B}$, by our choice of vertex redistribution, completing the proof of (ii).
	For (i), note that $|A\triangle A'|+|B\triangle B'|\leq |A_0|+|B_0|\leq \sqrt{\rho}n$. Now Lemma~\ref{lem:stillbiprobustexp}(i) with $\rho,\sqrt{\rho},\nu,\tau,A,B,A',B'$ playing the roles of $\rho,\gamma,\nu,\tau,A,B,A',B'$ shows that $U$ is a bipartite $(3\sqrt{\rho},\nu/2,2\tau)$-robust expander component with bipartition $A'$, $B'$, which completes the proof of (i).
\end{proof}


Finally, we can prove the existence of a polynomial-time algorithm to find a robust partition in regular graphs. Again, we follow the proof of \cite{kuhn2014robust} closely, but must suitably apply the algorithms developed in the previous section.

\begin{theorem}
	\label{th:decompose}
	For every $0< \tau <\alpha<1$ and every non-decreasing function $f:(0,1) \rightarrow (0,1)$ there is a $n_0$ and a polynomial-time algorithm that does the following. Given an $n$-vertex $D$-regular graph $G$ as input with $n\geq n_0$ and $D\geq \alpha n$, the algorithm finds a robust partition $\mathcal{V}$ with parameters $\rho,\nu,\tau,k,\ell$ with $1/n_0 < \rho < \nu < \tau$; $\rho<f(\nu)$, and $1/n_0 < f(\rho)$. 
\end{theorem}

%
%

\begin{proof}
	Set $t=\lceil 2/\alpha\rceil$\COMMENT{VP: why do we need a factor of 3;FS: \textbf{resolved}: the factor 3 and the subject of the next comment stem from the fact that the iteration step they use is formulated differently and can end up not changing the number of components.}. Define constants satisfying
	$$ 0 <1/n_0 \ll \rho_1 \ll \nu_1 \ll  \rho_2 \ll \nu_2 \ll \dots \ll \rho_t \ll \nu_t  \ll \tau' \ll\tau\leq \alpha.$$
	We start with the following claim:
	
	\begin{claim}
	\label{cl:partition}
		There is some $1\leq h<t$ and a partition $\mathcal{U}$ of $V(G)$ such that, for each $U\in \mathcal{U}$,  $U$ is a $(\rho_h, \nu_h, \tau')$-robust expander component or a bipartite $(\rho_h, \nu_h, \tau')$-robust expander component. Furthermore, we can find $\mathcal{U}$ in polynomial time (and we can determine those $U \in \mathcal{U}$ that are bipartite robust expander components together with a corresponding bipartition).
	\end{claim}
	
	We will iteratively construct (in polynomial time) a partition $\mathcal{U}_i$ of $V(G)$ such that $U$ is a $\rho_i$-component for all $U\in \mathcal{U}_i$. 
	
	We know $V(G)$ is a $\rho_1$-component for any choice of $\rho_1 > 0$ and we set $\mathcal{U}_1= \{V(G)\}$. 

	Assume that for some $1\leq i\leq t$ we have constructed such a partition $\mathcal{U}_i$ of $V(G)$.
We apply Algorithm 4 to each $U \in \mathcal{U}_i$ with $\rho_i, \nu_i, \rho_{i+1}, \tau'$ playing the roles of $\rho, \nu, \rho', \tau$. If the algorithm finds some $U \in \mathcal{U}_i$ for which it returns $U_1, U_2$, a partition of $U$ in which $U_1$ and $U_2$ are $\rho_{i+1}$- components, then 
we set  $\mathcal{U}_{i+1}:=(\mathcal{U}_i\setminus \{U\})\cup\{U_1,U_2\}$ and we continue.
Otherwise the algorithm determines that $G[U]$ is a robust $(\nu_i, \tau')$-expander or a bipartite robust $(\nu_i, \tau')$-expander for all $U \in \mathcal{U}_i$ and so each $U \in \mathcal{U}_i$ is a $(\rho_i, \nu_i, \tau')$-robust expander component or a bipartite $(\rho_i, \nu_i, \tau')$-robust expander component (and Algorithm 4 is able to determine which $U \in \mathcal{U}_i$ are bipartite robust expander components and to determine a corresponding bipartition $A,B$ of any such $U$). In this case we are done with the claim provided $i<t$, which we now show.
	
	By induction $|\mathcal{U}_{i+1}| = i+1$
\COMMENT{VP: in the paper they show that $2|\mathcal{U}_{i+1}| + |\mathcal{W}_{i+1}| \geq i+2$ but I'm not sure why it's needed. Please check. FS: \textbf{resolved}, see last comment.}	
	  and all $U\in \mathcal{U}_{i+1}$ are $\rho_{i+1}$-components whenever $\mathcal{U}_{i+1}$ is defined.
	To see that the process terminates before $\mathcal{U}_t$, assume for the sake of contradiction that $\mathcal{U}_t$ is defined. Since every $U\in \mathcal{U}_t$ is a $\rho_t$-component, $|U|\geq (\alpha-\sqrt{\rho_t})n$ for all $U\in\mathcal{U}_t$ by Proposition~\ref{Claim1}, and so
	$$n = |V(G)|\geq t(\alpha-\sqrt{\rho_t})n\geq \frac{2}{\alpha}(\alpha-\sqrt{\rho_t})n>n ,$$
	a contradiction, proving the claim.
	
	\vspace{0.3cm}
	So in polynomial time, we can find $\mathcal{U}=\{U_1,\dots,U_k,Z_1,\dots,Z_\ell\}$ for some $k, \ell \in \mathbb{N}$, where $U_i$ is a $(\rho',\nu',\tau')$-robust expander component for all $1\leq i\leq k$ and $Z_j$ is a bipartite $(\rho',\nu',\tau')$-robust expander component for all $1\leq j\leq \ell$, where  $\rho'=\rho_h,\nu'=\nu_h$ for some $h < t$. Furthermore our algorithm determines which $U \in\mathcal{U}$ are bipartite robust expander components and gives corresponding bipartitions for them.

	From Proposition~\ref{Claim1} and Remark~\ref{rem1} we know that $|U_i|\geq (D-\sqrt{\rho' }n
	)$ for $1\leq i\leq k$ and $|Z_j| \geq 2(D-2\sqrt{\rho'}n)$ for $1\leq j\leq \ell$. Therefore
	$$n=\sum_{1\leq i\leq k} |U_i| + \sum_{1\leq j \leq l} |W_j|\geq (D-2\sqrt{\rho'}n)(k+2\ell)$$
	and so
	\begin{equation}
	\label{eq:D6}
	k+2\ell  \leq \left\lfloor \frac{n}{D-2\sqrt{\rho'}n}\right\rfloor \leq \left\lfloor (1+\rho'^{1/3})\frac{n}{D}\right\rfloor.
	\end{equation}
 In particular  $m:= k+ \ell \leq (k + 2 \ell)  \leq 2n/D \leq 2\alpha^{-1}$. 	
	Now we apply the algorithm of Lemma~\ref{lem:modifycomp} (with $\rho'$ playing the role of $\rho$) to $\mathcal{U}$ to obtain (in polynomial time) the partition $\mathcal{V}=\{V_1,\dots,V_k,W_1,\dots,W_\ell\}$ of $V(G)$ satisfying (i)-(iv) so that in particular 
	\[
	|U_i \triangle V_i|, |Z_i \triangle W_i| \leq \rho'^{1/3}n \leq \nu' n
	\]
	for all applicable $i$ and $j$.	
	We now show that $\mathcal{V}$ is a $(\rho,\nu,\tau)$-robust partition of $G$, where $\rho = 3^{3/2}\rho'^{1/6}$, $\nu=\nu'/4$. Note that $\rho \leq f(\nu)$ by making a suitable choice of $\rho_i \ll \nu_i$ for each $i$ at the start. Similarly, a suitable choice of $\rho_1$ guarantees that $1/n_0 \leq f(\rho)$.
	
	Obviously (D1) holds. 
	For (D2), note that $V_i$ is a $\rho'^{1/3}$-component by Lemma~\ref{lem:modifycomp}(ii). As $\rho'^{1/3}\leq \rho$ and $|V_i| \geq D/2 \geq \sqrt{\rho}n$ (by Proposition~\ref{Claim1}), $V_i$ is a $\rho$-component. By Lemma~\ref{lem:modifycomp}(i) and Lemma~\ref{lem:stillrobustexp} with $\nu',\tau',U_i,V_i$ playing the roles of $\nu,\tau,U,U'$, we have that $G[V_i]$ is a robust $(\nu'/2,2\tau')$-expander and thus also a robust $(\nu,\tau)$-expander. This shows (D2).
	To show (D3), recall that $G[Z_j]$ is a bipartite $(\rho',\nu',\tau')$-robust expander component and our algorithm gives us a partition $A_j', B_j'$ of $Z_j$ demonstrating this. We obtain a partition $A_j'',B_j''$ of $W_j$ by taking $A_j'' = A_j' \cap W_j$ and $B_j'' = W_j \setminus A_j''$ so that $|A_j'' \triangle A_j'| + |B_j'' \triangle B_j'| \leq |Z_j \triangle W_j| \leq \rho'^{1/3}n $.	Then Lemma~\ref{lem:modifycomp}(ii) together with Lemma~\ref{lem:stillbiprobustexp}(i) where $\rho',\rho'^{1/3},\nu',\tau',Z_j,W_j$ play the roles of $\rho,\gamma,\nu,\tau,U,U'$ imply that $G[W_j]$ is a bipartite $(3\rho'^{1/3},\nu'/2,2 \tau')$-robust expander component. Next we apply (the algorithm of) Lemma~\ref{lem:modifybip} with $(3\rho'^{1/3},$ $\nu'/2,2\tau',W_j, A_j'', B_j'')$ playing the roles of $(\rho,\nu,\tau, U, A, B)$ to obtain a bipartition $A_j,B_j$ of $W_j$ (in polynomial time). Now (D3) follows from Lemma~\ref{lem:modifybip}(i). 
	We find that (D4) follows from Lemma~\ref{lem:modifycomp}(iii) and (D5) follows from Lemma~\ref{lem:modifybip}(ii).
	Lastly, (D6) follows from \eqref{eq:D6}
	and (D7) follows from Lemma~\ref{lem:modifycomp}(iv).
\end{proof}

\begin{remark}
	The running time of the algorithm of Theorem~\ref{th:decompose} is bounded by $O(n^4\alpha^{-2})$ where $n=|V(G)|$.
Indeed, examining the proof of Theorem~\ref{th:decompose}, the algorithm in Claim~\ref{cl:partition} makes $O(t^2) = O(\alpha^{-2})$ calls to algorithm 4. Algorithm 4 makes a single call to each of Algorithms 1,2,3, and each of these algorithms requires at most $n$ applications of either Theorem~\ref{th:cheeger} or Theorem~\ref{th:trevisan}, i.e.\ a total running time of $O(\alpha^{-2}) \cdot n \cdot O(n^2) = O(\alpha^{-2}n^3)$. This dominates the running time as the application of the (greedy) algorithms in Lemma~\ref{lem:modifycomp} and Lemma~\ref{lem:modifybip} runs in time $O(n^3)$.
	\end{remark}

%
\section{Finding almost-Hamilton cycles}
\label{sec:findingcycles}
In this section we show how to algorithmically determine whether a dense, regular graph $G$ has a very long cycle (missing at most a constant number of vertices) and how to construct such a cycle if it exists. The idea is that we first use the algorithm of the previous section to find a robust partition $\mathcal{U} = \{U_1, \ldots, U_m\}$ of our input dense regular graph. Then we try to find a \textit{path system} $\mathcal{P}$ (defined below) that supplies all the edges of our desired cycle between the $U_i$.\footnote{If $U_i$ is a bipartite robust component with bipartition $A_i, B_i$ then $\mathcal{P}$ may contain edges from $G[A_i]$ and $G[B_i]$ but not from $G[A_i,B_i]$} What properties should the edges in such a path system have? For any (almost) Hamilton cycle $H$ of $G$, the edges of $H$ between the $U_i$ should connect up the $U_i$ in some sense; thus the path system $\mathcal{P}$ should be \textit{connecting}, which we define precisely below. The path system should also be balancing in some sense: if $U_i$ is a bipartite component with parts $A_i$ and $B_i$ then the edges of $H \cap G[A_i,B_i]$ hit an equal number of vertices from $A_i$ and $B_i$, so the remaining edges of $H$ (namely those of $\mathcal{P}$) should counter any imbalance in the sizes of $A_i$ and $B_i$. It turns out that $G$ has a Hamilton cycle if and only if there is a connecting, balancing path system (with respect to $\mathcal{U}$). This was established in \cite{kuhn2014robust}; see Lemma~\ref{lem:pathsystemtocycle} below, which uses robust expansion to connect a connecting, balancing path system into a Hamilton cycle. Furthermore, it was shown in \cite{CyclePartitions} that a balancing path system always exists for dense, regular graphs.

Thus the problem of deciding (almost) Hamiltonicity reduces to the problem of deciding the existence of a connecting path system. We show how to determine this in polynomial time, which relies on the fact that the number of parts in $\mathcal{U}$ is finite.

The constant number of vertices that our cycle might miss owes to the fact that it is not always possible to combine balancing and connecting path systems perfectly. Nonetheless, we shall see that a very long cycle exists if and only if there is a connecting path system.

%


\subsection{Preliminaries}
In this subsection, we recall some definitions and results that will be used later. We begin by defining the structure required between the parts of our robust partition that ensures a Hamilton cycle.

A \emph{path system} $\mathcal{P} = \{P_1, \ldots, P_k \}$ in a graph $G$ is a collection of vertex-disjoint paths $P_1, \ldots, P_k$ in $G$. We also think of $\mathcal{P}$ as a subgraph $\mathcal{P} = \cup P_i \subseteq G$, so that $V(\mathcal{P})$ and $E(\mathcal{P})$ make sense.

\vspace{0.2 cm}
\noindent
{\bf Reduced graphs} - 
	Let $G$ be a graph and $\mathcal{U}$ a partition of $V(G)$. For a path system $\mathcal{P}\subseteq E(G)$ we define the \emph{reduced multigraph} $R_\mathcal{U}(\mathcal{P})$ of $\mathcal{P}$ with respect to $\mathcal{U}$ to be the multigraph with vertex set $\mathcal{U}$ and where there is an edge between $U,U'\in \mathcal{U}$ for each path in $\mathcal{P}$ whose endpoints are in $U$ and $U'$. We also define the \emph{reduced edge multigraph} $R'_\mathcal{U}(\mathcal{P})$ of $\mathcal{P}$ with respect to $\mathcal{U}$ as the multigraph with vertex set $\mathcal{U}$ and where there is an edge between $U,U'\in \mathcal{U}$ for each \emph{edge} in $\mathcal{P}$ with endpoints in $U,U'$. Note that both $R_\mathcal{U}(\mathcal{P})$ and $R'_\mathcal{U}(\mathcal{P})$ may contain loops and multiedges. We will often identify edges in $R_\mathcal{U}(\mathcal{P})$ (resp.\ $R'_\mathcal{U}(\mathcal{P})$) with their corresponding paths (resp.\ edges) in $\mathcal{P}$. We sometimes write $R(\mathcal{P})$ or $R'(\mathcal{P})$ if $\mathcal{U}$ is clear form the context.

\vspace{0.2 cm}
\noindent
{\bf Connecting and balancing path systems} -
	Let $G$ be a graph and $\mathcal{U}$ a partition of $V(G)$. A path system $\mathcal{P}\subseteq G$ is called \emph{$\mathcal{U}$-connecting} if $R_\mathcal{U}(\mathcal{P})$ is Eulerian, that is if $R_\mathcal{U}(\mathcal{P})$ is connected and all vertices have even degree.

	Let $A,B\subseteq V(G)$ be two disjoint sets. 
We say $\mathcal{P}$ is \emph{$k$-almost $(A,B)$-balancing} if 
$$\left|(|A|-e_\mathcal{P}(A,\overline{A\cup B})-2e_\mathcal{P}(A))-(|B|-e_\mathcal{P}(B,\overline{A\cup B})-2e_\mathcal{P}(B))\right|\leq k$$ 
and we say $\mathcal{P}$ is \emph{$(A,B)$-balancing} if it is $0$-almost $(A,B)$-balancing. The significance of this is that, given any cycle $C$ of $G$ that covers all vertices of $A \cup B$, if we delete from $C$ all edges of $E_G(A,B)$, the resulting path system will be $(A,B)$-balancing.

	
	
	For a robust partition $\mathcal{V}=\{V_1,\dots,V_k,W_1,\dots,W_\ell\}$ of $G$ where $A_j,B_j$  is the corresponding bipartition of $W_j$ for $1\leq j\leq \ell$, we say $\mathcal{P}$ is $\mathcal{V}$-balancing if it is $(A_i,B_i)$-balancing for $1\leq i \leq \ell$, and we say $\mathcal{P}$ is $k$-almost $\mathcal{V}$-balancing if it is $k_i$-almost $(A_i,B_i)$-balancing for $1\leq i\leq \ell$ and $\sum_{i=1}^{\ell}k_i\leq k$. The \emph{$\mathcal{V}$-imbalance} of $\mathcal{P}$ is the smallest $k$ for which $\mathcal{P}$ is $k$-almost $\mathcal{V}$-balancing. We will omit $\mathcal{V}$ if it is clear from context.
	
The definitions introduced so far have been for $\mathcal{U}$ a partition of $V(G)$, but they extend in the obvious way when $\mathcal{U}$ is a subpartition of $V(G)$, i.e.\ where $\mathcal{U}$ consists of disjoint subsets of vertices that do not necessarily cover all of $V(G)$ (and where it is implicitly assumed that $V(\mathcal{P}) \subseteq \cup_{U \in \mathcal{U}}U$).



\begin{lemma}[Lemmas 7.8 and 6.2 in \cite{kuhn2014robust}]
	\label{lem:pathsystemtocycle}
	
	Let $n,k,\ell\in \mathbb{N}_0$ and $0<1/n\ll\rho\ll\nu\ll\tau\ll\eta<1$. Let $G$ be a graph on $n$ vertices and suppose that $\mathcal{V}:=\{V_1,\dots,V_k,W_1,\dots,W_\ell\}$ is a weak robust subpartition of $G$ with parameters $\rho,\nu,\tau,\eta,k,\ell$. For each $1\leq j\leq \ell$, let $A_j,B_j$ be the bipartition of $W_j$. If $\mathcal{P}$ is a $\mathcal{V}$-connecting, $\mathcal{V}$-balancing path system such that $|V(\mathcal{P})\cap X| \leq \rho n$ for all $X\in \mathcal{V}$ then there is a cycle $C$ in $G$ that contains every vertex in $\cup_{U \in \mathcal{V}}U$. Furthermore there is a polynomial-time algorithm for constructing such a cycle.
\end{lemma}

\begin{remark}
	\label{rem:cyclequickenough}
	Lemma~\ref{lem:pathsystemtocycle} follows directly from Lemmas 7.8 and 6.2 in \cite{kuhn2014robust}. We do not state these results because their statements involve extraneous definitions not required for our purposes. Instead we briefly discuss the relevant results informally and how to make them algorithmic.
	
	In this paper, our definition of $\mathcal{V}$-balancing is different from that used in \cite{kuhn2014robust}. Lemma 7.8 from \cite{kuhn2014robust} is used to show that a path system $\mathcal{P}$ satisfying the conditions of Lemma~\ref{lem:pathsystemtocycle} can be used to construct a so-called $\mathcal{V}$-tour, which satisfies their stronger definition of balance. The proof is constructive and easily gives a polynomial-time algorithm for constructing such a $\mathcal{V}$-tour.
	Lemma 6.2 then shows how, given a $\mathcal{V}$-tour, one can construct a cycle $C$ as in Lemma~\ref{lem:pathsystemtocycle}. The proof shows explicitly how to reduce this problem to that of finding a Hamilton cycle in a robust $(\nu, \tau)$-expander. While they appeal to their Theorem 6.7, we can do this in polynomial time by appealing to Theorem 5 in \cite{christofides2012finding}.\NEWCHANGE{added the part that it's in $NC^5$. I'm pretty sure we can't get anything better than polynomial from $NC^5$ without going through their algorithm in detail, so this explains why we don't give more detail. VP: I removed $NC^5$ because we don't explain it. I think we can leave it like this.}
\COMMENT{VP: check the whole remark is clear and accurate. FS: \textbf{resolved:} checked} \COMMENT{VP: maybe it's a good idea to define $\mathcal{V}$-tour just as in \cite{kuhn2014robust} paper except replacing (T4) with \eqref{eq:balancing}. Then refer to $\mathcal{V}$-tour in Lemma~\ref{lem:pathsystemtocycle}. Then we can more easily refer to this definition later too.}
	
%
%
%
	
\end{remark}

Next\COMMENT{VP: Everything from here to end of section not fully checked yet. - \textbf{resolved}} we will state the results from \cite{CyclePartitions} that allow one to find balancing path systems in dense regular graphs. Their setup is different from \cite{kuhn2014robust}, so we now introduce the necessary definitions.


\vspace{0.2 cm}
\noindent
{\bf $\alpha$-sparse and $\alpha$-far from bipartite} -  
	Let $G$ be a graph on $n$ vertices. A \textit{cut} of a set $A\subseteq V(G)$ is a partition $X,Y$ of $A$, where $X$ and $Y$ are both non-empty. We say that a cut $X,Y$ is \textit{$\alpha$-sparse} if $e_G(X,Y)\leq \alpha |X||Y|$. We say that a set $A\subseteq V(G)$ is \textit{$\alpha$-almost-bipartite} if there exists a partition $X,Y$ of $A$ such that $G[A]$ has at most $\alpha n^2$ edges that are not in $E_G(X,Y)$. Otherwise, we say that $A$ is \textit{$\alpha$-far-from-bipartite}. 
	

\vspace{0.2 cm}
\noindent
{\bf Clustering} - 	Let $c_{min}\in (0,1)$ and let $G$ be a $D$-regular graph on $n$ vertices with $D\geq c_{min}n$. A \textit{clustering of $G$} with parameters $\zeta,\delta,\gamma,\beta,\eta$ is a partition $\{A_1,\dots,A_r\}$ of $V(G)$ into non-empty sets satisfying the following properties:
	\begin{enumerate}[noitemsep]
		\item[(a)] $G$ has at most $\eta n^2$ edges with ends in different $A_i$'s;
		\item[(b)] for each $i\in[r]$, the minimum degree of $G[A_i]$ is at least $\delta n$;
		\item[(c)] for each $i\in [r]$, $A_i$ has no $\zeta$-sparse cuts;
		\item[(d)] for each $i\in[r]$, $A_i$ is either $\beta$-almost bipartite or $\gamma$-far from bipartite. If $A_i$ is $\beta$-almost-bipartite, we also give an appropriate partition $X_i,Y_i$.
	\end{enumerate}
	We will always choose the parameters such that $1/n\ll \eta \ll \beta \ll \gamma\ll\zeta\ll \delta$.

The following theorem says that a clustering always has a balancing path system. Here we think of a path system as a subgraph of $G$.

\begin{theorem}[Lemma 5 in \cite{CyclePartitions}]
	\label{th:balance}
	
	Let $1/n \ll \eta \ll \beta \ll \xi,\gamma \ll \zeta \ll \delta < 1$.	
	Suppose $G$ is an $n$-vertex, $D$-regular graph with $D \geq c_{min}n$ and $\mathcal{A} = \{A_1,\dots,A_r\}$ is a clustering of $G$ with parameters $\zeta,\delta,\gamma,\beta,\eta$, and assume that whenever $A_i$ is $\beta$-almost-bipartite the corresponding partition of $A_i$ is $X_i,Y_i$. Then there exists a path system $H\subseteq G$ with the following properties:
	
	\begin{enumerate}[noitemsep]
		\item[(a)] For each $i\in[r]$ such that $A_i$ is $\beta$-almost-bipartite, we have
		$$ 2e_H(X_i)-2e_H(Y_i)+e_H(X_i,\overline{A_i})-e_H(Y_i,\overline{A_i})=2(||A_j|-|B_j||);$$
		\item[(b)] The number of leaves (i.e.\ vertices of degree $1$) of $H$ in $A_i$ is even for all $1\leq i\leq r$;
		\item[(c)] $|V(H)|\leq \xi n$.
	\end{enumerate}
	Furthermore, there is a randomized algorithm that finds $H$ with probability $p>\frac{3}{4}$ and runs in time polynomial in $n$.
\end{theorem}
\begin{remark}
	\label{rem:balsmall}
Note firstly that (a) says that $H$ is an $\mathcal{A}$-balancing path system. We shall see in the next lemma that a robust partition is a clustering, so this gives us a way of obtaining balancing path systems for robust partitions. 
	
	The lemma above is not stated to be algorithmic in \cite{CyclePartitions}, but in fact their probabilistic proof essentially gives a (randomised) polynomial-time algorithm. 
	Also, their proof requires that the probability $p$ of success be positive, but the analysis can easily be modified to show a lower bound of e.g.\ $p>\frac{3}{4}$.
\end{remark}


As Theorem~\ref{th:balance} uses the concept of a clustering, we use the following lemma to show that a robust partition is also a clustering. This allows us to apply Theorem~\ref{th:balance} to a robust partition.\COMMENT{\textbf{resolved} move following lemma to previous section or to an appendix}
\begin{lemma}
	\label{lem:cluster}
	For every non-decreasing function $f: (0,1) \to (0,1)$ there is a non-decreasing function $f': (0,1) \to (0,1)$ satisfying $f'(x)<f(x)$ for all $x \in (0,1)$ such that the following holds. For any choice of parameters 
$\rho, \nu, \tau, \alpha, n, k, \ell$ satisfying 	
	 $1/n \leq \rho \ll_{f'} \nu \leq \tau \ll_{f'} \alpha$ and $n, k, \ell \in \mathbb{N}$ there exist parameters $\zeta,\delta,\gamma,\beta,\eta$ satisfying 
$\rho \ll_f \eta \ll_f  \beta \ll_f \gamma \ll_f \zeta \ll_f \nu $	
	and $\tau < \delta < \alpha$ 
	such that if $G$ is an $n$-vertex $D$-regular graph with $D\geq \alpha n$ and $\mathcal{V}$ is a robust partition of $G$ with parameters $\rho, \nu,\tau,k,\ell$ then $\mathcal{V}$ is also clustering with parameters $\zeta,\delta,\gamma,\beta,\eta$. 
\end{lemma}
\begin{remark}
A proof of the above lemma is provided in the appendix for completeness. 

\end{remark}

\subsection{Path systems and long cycles}

The first lemmas in this subsection, \ref{le:redCsystem} to \ref{lem:connect} show how to find connecting path systems. The rest of the chapter shows how to combine all the elements. Lemma~\ref{lem:combine} allows us to combine balancing and connecting path systems into a single path system that is connecting and almost balancing, and Lemma~\ref{lem:cycle} allows us to extend this path system into a very long cycle (by applying Lemma~\ref{lem:pathsystemtocycle}). At the end of the section comes the proof of Theorem~\ref{th:result}, which describes the whole algorithm.

\begin{lemma}
\label{le:redCsystem}
Let $G$ be a graph, let $\mathcal{U}=\{U_1,\dots,U_m\}$ be a partition of $V(G)$, and let $\mathcal{C}$ be a $\mathcal{U}$-connecting path system in $G$. Then there exists a $\mathcal{U}$-connecting path system $\mathcal{C}'$ such that 
\begin{itemize}[noitemsep]
\item[(a)]$E(\mathcal{C}') \cap E(G[V_i]) = \emptyset$ for all $i = 1, \ldots, m$ and 
\item[(b)]$|E(\mathcal{C}') \cap E_G(V_i,V_j)| \leq 2$ for all $1 \leq i < j \leq m$. 
\end{itemize}
\end{lemma}
\begin{proof}
	For any path $P=v_1v_2\cdots v_j$ in $\mathcal{C}$, if two vertices of $P$ belong to the same component $U \in \mathcal{U}$, let $v_a$ and $v_b$ be the first and last vertices of $P$ that belong to $U$ and replace $P$ with the paths $v_1Pv_a$ and $v_bPv_j$; it is easy to see that the resulting path system is $\mathcal{U}$-connecting (see Figure~\ref{fig:breakpaths}). We make replacements as described above until no paths contain multiple vertices from the same component and we call the resulting $\mathcal{U}$-connecting path system $\mathcal{C}^*$. 
	\begin{figure}
		\centering
		\includegraphics[width=0.45\textwidth]{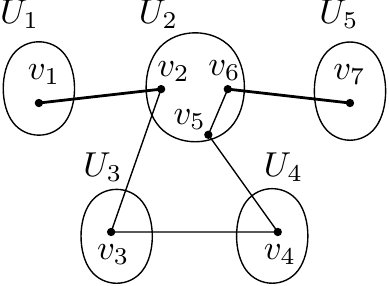}
		\caption{Example: The path $v_1\dots v_7$ from $U_1$ to $U_5$ has the edges between $v_2$ and $v_6$ pruned, resulting in two paths (thick lines), one from $U_1$ to $U_2$ and one from $U_2$ to $U_5$. Note that this ensures that $\mathcal{C}'$ contains no edges inside components.}
		\label{fig:breakpaths}
	\end{figure}	
	Next we show how to reduce the number of edges between components.
	\begin{claim}
Let $\mathcal{D}$ be a $\mathcal{U}$-connecting path system (i.e.\ $R_{\mathcal{U}}(\mathcal{D})$ is Eulerian).
		For $X,Y\in \mathcal{U}$ such that $E_{\mathcal{D}}(X,Y)>2$, it is possible to find two edges $e,f \in E_{\mathcal{D}}(X,Y)$ such that $\mathcal{D}' = \mathcal{D} \setminus \{e,f \}$ is a $\mathcal{U}$-connecting path system. (Here deleting $e, f$ from $\mathcal{D}$ may create isolated vertices which we remove to form	 $\mathcal{D} \setminus \{e,f \}$.)
	\end{claim}
	\begin{proof}[Proof of claim]
	We first note that if $e \in E_{\mathcal{D}}(X,Y)$, then the effect of deleting $e$ from $\mathcal{D}$ is to keep all degrees of $R(\mathcal{D})$ unchanged except that the degrees of $X$ and $Y$ will increase or decrease by $1$. (Note that we only get a decrease by $1$ if $e$ is the first or last edge of a path in $\mathcal{D}$.) Therefore removing two edges of $E_{\mathcal{D}}(X,Y)$ from $\mathcal{D}$ preserves the parity of all vertices of $R(\mathcal{D})$.
	
	Next suppose that $R(\mathcal{D})$ is Eulerian (and hence connected).   Hence $R(\mathcal{D})$ is in fact $2$-edge connected\COMMENT{add the definition somewhere FS: this is the only place edge connectivity appears. Are we assuming the reader is familiar with Menger's theorem, but not edge connectivity? \textbf{resolved} - I added a reference to Diestel for any notation not covered} (since an Eulerian graph can be decomposed into cycles but	a cut edge cannot belong to a cycle). Therefore by Menger's theorem there are two edge-disjoint paths $Q_1$ and $Q_2$ between $X$ and $Y$ in $R(\mathcal{D})$. Given any three edges of $E_{\mathcal{D}}(X,Y)$, we can find two, say $e,f$, that miss either $Q_1$ or $Q_2$, say $Q_1$. 
	
	Let $P_e$ be the path of $\mathcal{D}$ containing $e$. 
	The effect on $R(\mathcal{D})$ of removing $e$ from $\mathcal{D}$ is to replace some edge $AB$ with two edges $AX, BY$.\footnote{It does not affect what follows, but strictly speaking, if $e$ is the first (resp.\ last) edge of $P_e$ then $AX$ (resp.\ $BY$) is a loop and is not present in $R(\mathcal{D} \setminus \{e\})$.} Therefore $A$ and $B$ are still connected in $R(\mathcal{D} \setminus \{e\})$ via the path $AXQ_1YB$. Similarly, deleting $f$ keeps the reduced graph connected. Therefore $R( \mathcal{D} \setminus \{e,f\})$ is connected with all degree parities preserved, so is Eulerian, i.e.\ $\mathcal{D}'= \mathcal{D} \setminus \{e,f\}$ is a connecting path system.
	\end{proof}
		
	We construct $\mathcal{C}'$ from $\mathcal{C}^*$ by iteratively applying the previous claim whenever possible. By construction $\mathcal{D}$ is a $\mathcal{U}$-connecting path system satisfying (a) and (b).
\end{proof}

The next lemma will be useful in our algorithm for detecting graphs that do not have very long cycles. It essentially says that the absence of a $\mathcal{U}$-connecting path system implies the absence of a very long cycle.

\begin{lemma}
	\label{lem:shortConnecting}
	Let $G$ be a graph and $\mathcal{U}=\{U_1,\dots,U_m\}$ be a partition of $V(G)$. If there exists a cycle $K$ in $G$ that contains at least $r>2m$ vertices from each $U\in \mathcal{U}$, then there also exists a $\mathcal{U}$-connecting path system $\mathcal{C}$ with at most $m^2-m$ edges. Further, $\mathcal{C}$ contains at most two edges between any two $U_i,U_j\subseteq \mathcal{U}$.
\end{lemma}
\begin{proof}
	
	We start by deleting edges from $K$ to form a path system $\mathcal{C}^*$ such that $R_{\mathcal{U}}(\mathcal{C}^*)$ is a Hamilton cycle on $\mathcal{U}$.	

\begin{claim}
There exist vertex-disjoint paths $P_1, \ldots, P_m \subseteq K$ such that the endpoints of $P_i$ are in $U_i$. 
\end{claim}
\begin{proof}[Proof of claim]
 Suppose, by induction, we have found vertex-disjoint paths $P_1, \ldots, P_{k-1}$ (with $k \leq m$) such that 
\begin{itemize}[noitemsep]
\item[(a)] each $P_i$ (with $i\leq k-1$) has its endpoints in $U_i$ (after relabelling of indices);
\item[(b)] $K \setminus (\cup_{i=1}^{k-1}V(P_i))$ is a union of paths that visits $U_i$ at least $r - (k-1)>m$ times for each $i \geq k$.
\end{itemize}
Any vertex of $\cup_{i=k}^m U_i$ is called \emph{untreated}.
We know that since $K$ is a cycle,  $K \setminus (\cup_{i=1}^{k-1}V(P_i))$ is a disjoint union of $k-1$ paths, which we denote by $Q_1, \ldots, Q_{k-1}$.  At least one of these paths, say $Q_1$ must contain at least $(r-k+1)(m-k+1)/(k-1) > m-k+1$ untreated vertices. Pick two untreated vertices $a, b \in V(Q_1)$ that are as close together as possible and belong to the same $U_j$ for some $j \geq k$. In particular, no two internal untreated vertices of $aQ_1b$ belong to the same $U_i$ and so $aQ_1b$ contains at most $ m-k+1 $ untreated vertices.  
Then we swap the indices of $U_j$ and $U_k$ and set $P_k = aQ_1b$. It is clear that (a) holds with $k-1$ replaced by $k$. Since, for each $i \geq k+1$, the path $P_k$ visits each $U_i$ at most once, part (b) also holds. (It is easy to see that a slight variant of the above argument allows us to pick the first path.)
\end{proof}

Let $\mathcal{C}^*$ be the set of non-trivial paths of $K \setminus \cup_{i=1}^m E(P_i)$; it is easy to see that $\mathcal{C}^*$ is a Hamilton cycle on $\mathcal{U}$ and so is a $\mathcal{U}$-connecting path system. Then, by the previous lemma applied to $\mathcal{C}^*$, there exists a $\mathcal{U}$-connecting path system $\mathcal{C}$ that has no edges inside any $U \in \mathcal{U}$ and that has at most $2$ edges between any distinct $U_i, U_j \in \mathcal{U}$ (and therefore has at most $m(m-1)$ edges).
	\end{proof}

The following lemma gives an algorithm for deciding whether a graph with vertex partition $\mathcal{V}$ has a $\mathcal{V}$-connecting path system.

\begin{lemma}
	\label{lem:connect}
	Let $G$ be a graph on $n$ vertices and $\mathcal{V}$ a partition of $V(G)$ with $|\mathcal{V}|=m$. There exists an algorithm that determines whether there exists a $\mathcal{V}$-connecting path system in $G$, and if one does, then the algorithm finds one with at most $m^2-m$ edges. This algorithm runs in time $m^{O(m^2)}+O(m^2 n^{5/2})$.
\end{lemma}


\begin{proof}
	The algorithm proceeds by first preselecting a small number of plausible edges and then using brute force to find a connecting path system as a subset of these edges.
The preselected edges are chosen such that if a $\mathcal{V}$-connecting path system exists, then one exists amongst the preselected edges. 
	
		
	Assume $\mathcal{V} = \{V_1, \ldots, V_m\}$. For each $1 \leq i < j \leq m$, let $E_{i,j} \subseteq E_G(V_i,V_j)$ be defined as follows.
	If the bipartite graph $G[V_i,V_j]$ contains a matching of size $4m$, let $E_{i,j}$ be the edges in any such matching. If not then $G[V_i,V_j]$ has a dominating set $F_{i,j}$ of size at most $8m$ (taking the vertices incident to a maximum matching). For each vertex $v$ in $F_{i,j}$, select any set $E_{i,j}^v$ of $\min(d_{G[V_i,V_j]}(v),2m)$ edges incident to $v$ in $G[V_i,V_j]$ and take $E_{i,j} = \cup_{v \in F_{i,j}}E_{i,j}^v$. Finally our preselected edge set is defined to be $E' := \cup_{i<j}E_{i,j}$. 	 
		
	Next we show that if a $\mathcal{V}$-connecting path system $\mathcal{C}$ exists, then also a $\mathcal{V}$-connecting path system $\mathcal{D}\subseteq E'$ exists.
 By Lemma~\ref{le:redCsystem} we may assume that $\mathcal{C}$ has no edges inside any $V_i$ and has at most two edges between each pair $V_i, V_j$ (so in particular there are at most $2(m-1)$ edges of $E(\mathcal{C})$ incident with $V_i$ (and $V_j$)).\COMMENT{VP: previous lemma assumes a cycle. FS: resolved, reference pointed to wrong lemma}	


\begin{claim}
Let $\mathcal{C}$ be any $\mathcal{V}$-connecting path system as described above, i.e.\ $\mathcal{C}$ has no edges inside any $V_i$ and has at most two edges between each pair $V_i, V_j$. Then for any $e \in E(\mathcal{C})$, we can find $r(e) \in E'$ such that
\begin{itemize}
\item[(R1)] if $e$ has its endpoints in $V_i$ and $V_j$, then so does $r(e)$;
\item[(R2)] for all $f \in E(\mathcal{C}) \setminus \{ e \}$,  if  $e \cap f = \emptyset$, then $r(e) \cap f = \emptyset$.
\end{itemize}
\end{claim}

We will repeatedly apply this claim to replace edges $e \in \mathcal{C}$ with edges $r(e) \in E'$ to obtain $\mathcal{D}$.

\begin{proof}[Proof of claim.]

In order to find $r(e)$ satisfying (R1) and (R2), assume $e$ has endpoints in $V_i$ and $V_j$. If $e \in E'$ then set $R(e) = e$ and note that (R1) and (R2) clearly hold. If not, then we have two cases to consider.

	 If  $E_{i,j}$ is a matching of size $4m$ then at least one edge of $E_{i,j}$ is not incident with any edge in $E(\mathcal{C})$ (since there are at most $2(m-1)$ edges of $\mathcal{C}$ incident with any $V_i$) and this is the edge we choose as $r(e)$; clearly (R1) and (R2) hold in this case. 

	 If $E_{i,j}$ is not a matching of size $4m$, then $e$ is incident to some vertex  $v \in F_{i,j}$, so assume $e = vv'$ and that $v \in V_i$ and $v' \in V_j$. Since $e \not\in E_{i,j}$, then $E_{i,j}$ has $2m$ edges incident to $v$, and so there is at least one edge $vv^* \in E_{i,j}$  such that $v^*$ is not incident to any edge in $E(\mathcal{C})$ (again since there are at most $2(m-1)$ edges of $\mathcal{C}$ incident to $V_j$),
and we choose $r(e) = vv^*$. Again (R1) and (R2) follow by construction.
\end{proof}

We now apply the above claim to $\mathcal{C}$, replacing each edge $e \in E(\mathcal{C})$ with $r(e)$ one at a time (each time updating $\mathcal{C}$ before the next application of the claim). Denote the resulting set of edges by $\mathcal{D}$.
Note that $E(\mathcal{D}) \subseteq E'$ and	
	 \begin{itemize}[noitemsep]
	 \item[(a)] if $e \in E(\mathcal{C})$ has its endpoints in $V_i$ and $V_j$, then so does $r(e) \in \mathcal{D}$;
	 \item[(b)] if $e,f \in \mathcal{C}$ are independent (i.e. $e \cap f = \emptyset$) then so are $r(e)$ and $r(f)$.
	 \end{itemize}
Here (b) holds because (R2) guarantees we never introduce any new incidences during the process of replacing edges.	 
	  

	   It is easy to see from (b) that $\mathcal{D}$ is a path system, and we now check that $\mathcal{D}$ is $\mathcal{V}$-connecting.  By (a) and (b), for any path $P \in \mathcal{C}$, the set of edges $\{r(e): e \in E(P)\}$ is a union of vertex-disjoint paths $P_1, \ldots, P_t$ with $P_i = a_iP_ib_i$ and $a_{i+1}$ and $b_i$ belong to the same $V \in \mathcal{V}$. Therefore each edge $e = VV' \in R(\mathcal{C})$ corresponds to a path from $V$ to $V'$ in $R(\mathcal{D})$ (with edges $e_1, \ldots, e_t$ corresponding to the paths $P_1, \ldots, P_t$). This shows that $R(\mathcal{D})$ can be obtained from $R(\mathcal{C})$ by replacing each edge with a path having the same endpoints as the edge: it is now clear that if $R(\mathcal{C})$ is Eulerian then so is $R(\mathcal{D})$ and so $\mathcal{D}$ is $\mathcal{V}$-connecting.
	 
	We have now shown that if a $\mathcal{V}$-connecting path system exists, then one exists inside $E'$ (and we have seen that it uses at most $2$ edges between each $V_i,V_j$, so at most $m^2 - m$ edges in total).
	For the algorithm to find such a path system, we first construct each $E_{i,j}$; the running time here is dominated in searching for a maximum matching in each $G[V_i,V_j]$, which takes total time $\binom{m}{2}n^{2.5}$ (using e.g.\ the Hopcroft-Karp algorithm \cite{hopcroft2karp}). We then check every possible way of selecting at most two edges from each $E_{i,j}$; since $E_{i,j}$ has size at most $(8m)(2m) = 16m^2$, there are $\left(\binom{16m^2}{2}   +16m^2+1\right)^{\binom{m}{2}} = m^{O(m^2)}$ possibilities. If a $\mathcal{V}$-connecting path system exists, then one of these possibilities will give us one and it takes time $m^{O(m^2)}+O(m^2 n^{2.5})$-time to determine this.
\end{proof}

The next lemma allows us to combine a connecting path system with a balancing path system into a path system that is connecting and almost-balancing.

\begin{lemma}
	\label{lem:combine}
Given a graph $G$ on $n$ vertices with a robust partition $\mathcal{V}=\{V_1,\dots,V_k,W_1,\dots,W_\ell\}$, a $\mathcal{V}$-balancing path system $\mathcal{B}$ and a $\mathcal{V}$-connecting path system $\mathcal{C}$, there exists a connecting, $(5|E(\mathcal{C})| + m-1)$-almost balancing path system $\mathcal{P}$, where $m:=k+\ell$ is the number of components in $\mathcal{V}$, and $\mathcal{P} \subseteq \mathcal{B} \cup \mathcal{C}$ (when thought of as sets of edges). Furthermore, $\mathcal{P}$ can be constructed in time polynomial in $n$. (Note that we suppress the parameters of the robust partition as they are irrelevant for this lemma.)
\end{lemma}

\begin{proof}
	We begin by constructing $\mathcal{B}'\subseteq\mathcal{B}$ as follows: First delete any edge from $\mathcal{B}$ that shares a vertex with an edge from $\mathcal{C}$ to obtain $\mathcal{B}^*$. As each edge in $\mathcal{C}$ is incident to at most four edges in $\mathcal{B}$, we delete at most $4|E(\mathcal{C})|$ edges here. 
	\begin{claim}
		There exists $\mathcal{B}'\subseteq\mathcal{B}^*$ such that $|E(\mathcal{B}^*) \setminus E(\mathcal{B}')|\leq m-1$ and every vertex of $R_{\mathcal{V}}(\mathcal{B}')$ has even degree.
	\end{claim}
	\begin{proof}[Proof of claim]
		Consider a connected component $X$ of the multigraph $\mathcal{R}'_\mathcal{V}(\mathcal{B}^*)$. As in any graph, there are an even number of vertices with odd degree in $X$. For each component of $\mathcal{R}'_\mathcal{V}(\mathcal{B}^*)$, pair up these vertices arbitrarily and find  paths (not necessarily disjoint) between each pair within $ \mathcal{R}'_\mathcal{V}(\mathcal{B}^*)$ (which is possible since each pair belongs to the same connected component of  $\mathcal{R}'_\mathcal{V}(\mathcal{B}^*)$); call these paths $P_1, \ldots, P_t$. 
		Set $Q=\triangle_{i=1}^tP_i$ as the symmetric difference of the edge sets of $P_1,\dots, P_t$. Note that removing all edges in $Q$ from $\mathcal{R}'_\mathcal{V}(\mathcal{B}^*)$ will result in a graph with even degree in each vertex. Next, construct $Q'$ from $Q$ by iteratively removing edges that form cycles, where we count a double edge as a cycle. Do this until no cycles remain, i.e.\ $Q'$ is a forest so has at most $m-1$ edges. Again, removing the edges in $Q'$ from $\mathcal{R}'_\mathcal{V}(\mathcal{B}^*)$ results in a graph with even degree in each vertex. The edges in $Q'$ correspond to edges in $\mathcal{B}^*$ that we delete to construct $\mathcal{B}'$, and so $\mathcal{R}'(\mathcal{B}')$ has even degree in every vertex. As the parity of each degree in $\mathcal{R}_\mathcal{V}(\mathcal{B}')$ and $\mathcal{R}'_\mathcal{V}(\mathcal{B}')$ are the same, $\mathcal{R}_\mathcal{V}(\mathcal{B}')$ has even degree in each vertex. 
	\end{proof}
	
	We construct $\mathcal{P}$ as the union of $\mathcal{B}'$ and $\mathcal{C}$. Both $R_{\mathcal{V}}(\mathcal{B'})$ and $R_{\mathcal{V}}(\mathcal{C})$ have even degree for every vertex 
	 and so this also holds for $R_{\mathcal{V}}(\mathcal{P})$. Since $R_{\mathcal{V}}(\mathcal{C})$ is connected so is $R_{\mathcal{V}}(\mathcal{P})$ and so $\mathcal{R}_\mathcal{V}(\mathcal{P})$ is Eulerian, i.e.\ $\mathcal{P}$ is  $\mathcal{V}$-connecting. By construction $\mathcal{P}$ arises from $\mathcal{B}$ by at most $5|E(\mathcal{C})|+m-1$ additions or deletions of edges, each of which contributes at most 1 to the $\mathcal{V}$-imbalance of $\mathcal{P}$. It is straightforward to see that $\mathcal{P}$ can be constructed in time polynomial in $n$ given $G,\mathcal{B}, \mathcal{C}$.
\end{proof}

If we have a connecting, almost balancing path system (as provided by the previous lemma) with respect to a robust partition, then we can use Lemma~\ref{lem:pathsystemtocycle} to construct a very long cycle, as described below.

\begin{lemma}
	\label{lem:cycle} Let $0<1/n_0 \ll \rho\leq \gamma \ll \nu\ll \tau \leq \alpha<1$ and $t\leq \rho n$.
	There is an algorithm that, given an $n$-vertex, $D$-regular graph $G$ with $n \geq n_0$ and $D \geq \alpha n$ and a robust partition  $\mathcal{V}=\{V_1,\dots,V_k,W_1,\dots,W_\ell\}$ of $G$ with parameters $\rho,\nu,\tau,k,\ell$ and a $\mathcal{V}$-connecting $t$-almost balancing path system $\mathcal{P}$ with $|V(\mathcal{P})\cap V|\leq \gamma n$ for all $V\in\mathcal{V}$, constructs a cycle through all but at most $t$ vertices of $G$. It does this in time polynomial in $n$.
\end{lemma}

\begin{proof}

	We use Lemma~\ref{lem:weakrobustsubp} to see that $\mathcal{V}$ is also a weak robust subpartition with parameters $\rho,\nu,\tau,\eta,k,\ell$ where we set $\eta=\alpha^2/2$.
	
	For $1\leq j\leq \ell$, let $t_j$  be such that $\sum t_j =t$ and such that $\mathcal{P}$ is $t_j$-almost $(A_j,B_j)$-balancing, where $A_j,B_j$ is the bipartition corresponding to $W_j$. By selecting $t_j$ vertices $T_j$ from either $A_j\setminus V(\mathcal{P})$ or $B_j\setminus V(\mathcal{P})$, we can ensure that $\mathcal{P}$ is $(A_j\setminus T_j,B_j\setminus T_j)$-balancing. Set $T = \cup\, T_j$ so that $|T|=t\leq \rho n$ and define $\mathcal{V}'=\{V'_1,\dots,V'_k,W'_1,\dots,W'_\ell\}$ with $V'_i=V_i\setminus T = V_i$ and $W'_j=W_j\setminus T$ with $A_j\setminus T,B_j\setminus T$ as the bipartition of $W'_j$. 
	
	Next we show that $\mathcal{V}'$ is a weak robust subpartition of $G$ with parameters $3\gamma,\nu/2,2\tau,\alpha^2/4,k,\ell$. 
	
	First we apply Lemma~\ref{lem:stillbiprobustexp}(ii) to each $W_j$ with $W_j\setminus T$ playing the role of $U'$. As $|W_j\triangle W'_j|\leq \rho n \leq \gamma n$, we see that each $W_j$ is a bipartite $(3\gamma, \nu/2, 2\tau)$-robust expander component of $G$ (with bipartition $A_j\setminus T, B_j \setminus T$ by Lemma~\ref{lem:stillbiprobustexp}(i)). Clearly each $V_i' = V_i$ remains a $(\rho, \nu, \tau)$-robust expander component and so  is a $(3 \gamma, \nu/2,2 \tau)$-robust expander component as well. This shows that (D2$'$) and (D3$'$) hold. (D1$'$)  obviously holds, and as $|T|\leq \rho n$, it is easy to see that (D4$'$) and (D5$'$) also hold.
	
	To construct the desired cycle (i.e.\ one that contains every vertex of $V(G) \setminus T$), we apply Lemma~\ref{lem:pathsystemtocycle} with
	$G,3\gamma,\nu/2,2\tau,\alpha^2/4,n,k,\ell,\mathcal{V}',\mathcal{P}$ playing the roles of $G,\rho,\nu,\tau,\eta,n,k,\ell,\mathcal{V},\mathcal{P}$. We obtain a cycle $C$ that contains all vertices in $\cup_{X \in \mathcal{V}'}X = V(G) \setminus T$.	Moreover, this cycle can be found in time polynomial in $n$ since we can find $T$ in polynomial time and apply Lemma~\ref{lem:pathsystemtocycle} in polynomial time.
\end{proof}

Finally, we prove the main result, which we repeat here for convenience.
\begin{thm2}
	\label{th:result}
For every $\alpha\in(0,1]$,  there exists $c = c(\alpha) = 100 \alpha^{-2}$ and a (deterministic) polynomial-time algorithm that, given an $n$-vertex $D$-regular graph $G$  with $D\geq \alpha n$ as input, determines whether $G$ contains a cycle on at least $n - c$ vertices. 
Furthermore	there is a (randomised) polynomial-time algorithm to find such a cycle if it exists.	 
\end{thm2}

\begin{proof}

	We are given $\alpha$ in the statement of the theorem.
	We will choose non-decreasing functions $f_1,f_2,f_3,f_4: (0,1) \to (0,1)$ with $f_i(x)\leq x$ for all $x\in (0,1),i\in[4]$ as follows.
		Let $f_1$ be the function governing the hierarchy in the statement of Lemma~\ref{lem:cycle} 
and let $f_2$ be the function governing the hierarchy of Theorem~\ref{th:balance} 
		Define $f_3: (0,1) \to (0,1)$ as $f_3(x)=\min\{f_1(x),f_2(x), \alpha^2 x^2/100\}$.
	Applying Lemma~\ref{lem:cluster} with $f_3$ playing the role of $f$, let $f_4$ be the function we obtain (i.e.\ $f_4:=f'$) and note that $f_4(x) \leq f_3(x)$ for all $x \in (0,1)$.

	We define $\tau = f_4(\alpha)$ and apply Theorem~\ref{th:decompose} with $\tau,\alpha, f_4$ playing the roles of $\tau,\alpha,f$ to obtain a number $n_0 \in \mathbb{N}$. Define $c := 100\alpha^{-2}$. So far we have defined $f_1, \ldots, f_4, \tau, \alpha, n_0, c$.
	
	Given an $n$-vertex $D$-regular graph $G$ with $D \geq \alpha n$, if $n \leq \max(n_0, 1000\alpha^{-3})$ we can use brute force to determine in polynomial time if there exists a cycle in $G$ on at least $n - c$ vertices. So we assume that $n \geq \max(n_0, 1000\alpha^{-3})$.
	
	By applying Theorem~\ref{th:decompose} to $G$ (with $\tau, \alpha, n_0$ as above and $f=f_4$), we obtain a robust partition $\mathcal{V}$ of $G$ with parameters $\rho, \nu, \tau, k, \ell$ satisfying
\begin{equation}
\label{eq:hier1}
1/n_0 \ll_{f_4} \rho \ll_{f_4} \nu \leq \tau \ll_{f_4} \alpha.
\end{equation}
Set $m:= k+ \ell = |\mathcal{V}|$ and note that $m \leq (1 + \rho^{1/3})/ \alpha \leq 2 \alpha^{-1}$. 

We claim that $G$ contains a cycle with at least $n - c$ vertices if and only if $G$ has a $\mathcal{V}$-connecting path system. The claim proves the first part of the Theorem because, by applying the algorithm of Lemma~\ref{lem:connect}, we can determine in time polynomial in $n$ whether $G$ has a $\mathcal{V}$-connecting path system
(and if it does, we can find one in time polynomial in $n$ with at most $m^2$ edges).	

So let us prove the claim. First assume $G$ has no $\mathcal{V}$-connecting path system. Then by Lemma~\ref{lem:shortConnecting}, for every cycle $K$ of $G$, there is some $U \in \mathcal{V}$ such that $K$ contains at most $2m$ vertices of $U$; in particular $K$ misses at least
\[
|U| - 2m \geq (\alpha - \sqrt{\rho})n - 2m \geq (\alpha/2)n - 2m \geq  c
\]
vertices, where the first inequality is by Proposition~\ref{Claim1}, the second since $\rho \ll_{f_4} \alpha$ with $f_4(x) \leq f_3(x) \leq x^2/4$, and the third by our choice of $n$ large and $c$.

Now suppose $G$ contains a $\mathcal{V}$-connecting path system. Then we know there exists a $\mathcal{V}$-connecting path system $\mathcal{P}$ with at most $m^2$ edges. By Lemma~\ref{lem:cluster} with $f_3, f_4$ playing the roles of $f, f'$ and using \eqref{eq:hier1}, we see that $\mathcal{V}$ is a clustering with parameters $\zeta,\delta,\gamma,\beta,\eta$ where 
\begin{equation}
\label{eq:hier2}
1/n \ll_{f_3} \rho \ll_{f_3} \eta \ll_{f_3} \beta \ll_{f_3} \gamma \ll_{f_3} \zeta \ll_{f_3} \nu \leq \tau \leq \delta \leq \alpha.
\end{equation}	
Set $\xi:=\gamma$. In particular 
$n, \eta, \beta, \gamma, \xi, \zeta, \delta$ satisfy the hierarchy needed to apply Theorem~\ref{th:balance} to $G$  (with $\mathcal{V}, \alpha$ playing the roles of $\mathcal{A}, c_{\min}$). 
Thus there exists $H \subseteq G$ that is $\mathcal{V}$-balancing (by part (a)) and such that $|V(H)| \leq \xi n = \gamma n$ (by part (c)). Now applying Lemma~\ref{lem:combine} with $G, \mathcal{V}, H, \mathcal{P}$ playing the roles of $G, \mathcal{V}, \mathcal{B}, \mathcal{C}$, there exists a $\mathcal{V}$-connecting, $r$-almost balancing path system $\mathcal{P}' \subseteq \mathcal{P} \cup H$ where $r \leq 5|E(\mathcal{P})| +m -1 \leq 5m^2 + m \leq c$ (hence $\mathcal{P}'$ is also $c$-almost balancing). Note that for each $U \in \mathcal{V}$, we have $|V(\mathcal{P}') \cap U| \leq |V(H) \cap U| + |V(\mathcal{P})| \leq \xi n + 2m^2 \leq 2 \xi n$. 	
By Lemma~\ref{lem:cycle} with $G, \mathcal{V}, \mathcal{P}', \rho, 2\xi, \nu, \tau, \alpha, c $ playing the role of $G, \mathcal{V}, \mathcal{P}, \rho, \gamma, \nu, \tau, \alpha, t $, we see there exists a cycle $C$ in $G$ with at least $n - c$ vertices. We note that the required hierarchy for applying Lemma~\ref{lem:cycle} follows from \eqref{eq:hier2} and our choice of $f_3$ and it is also easy to see that $c \leq \rho n$ (since $1/n \ll_{f_3} \rho$ and our choice of $f_3$). This proves the claim.

Finally, if our algorithm determines that there exists a cycle in $G$ with at least $n - c$ vertices then there is also a polynomial-time algorithm to construct such a cycle. Indeed repeating the argument above with the corresponding algorithms, in polynomial time we can construct $\mathcal{P}$ (Lemma~\ref{lem:connect}) and $H$ (Theorem~\ref{th:balance} and Remark~\ref{rem:balsmall}) and therefore also $\mathcal{P}'$ (Lemma~\ref{lem:combine}) and hence also $C$ (Lemma~\ref{lem:cycle}).
\end{proof}

\begin{remark}
\label{rem:runtime}
 	The algorithm in Theorem~\ref{th:result} (for determining the existence of the cycle) has a crude running time upper bound of $O(\alpha^{-2}n^3) + O(\alpha^{-4}n^{5/2}) + g(\alpha)$, for some function $g$. Indeed $O(\alpha^{-2}n^3)$ comes from the application of Theorem~\ref{th:decompose} and Lemma~\ref{lem:connect}. The contribution of $g(\alpha)$ comes from using brute force when $n \leq \max(n_0, 100\alpha^{-3})$ and the application of Lemma~\ref{lem:connect}. 
 	
 	We do not give an explicit running time for finding the desired cycle (when it exists) because this algorithm is based on other polynomial-time algorithms in the literature where no explicit running time bound was given.
\end{remark}

\section{Conclusion}
\label{sec:conc}


The most obvious question that arises from this work is whether we can take $c=0$ in Theorem~\ref{th:result}, i.e.\ whether the Hamilton cycle problem is polynomial-time solvable for dense, regular graphs. Our work shows that to answer this affirmatively, it is enough to give a polynomial-time algorithm to decide whether there exists a path system that is both $\mathcal{V}$-connecting and $\mathcal{V}$-balancing when given a dense regular graph together with a robust partition $\mathcal{V}$.
 
 One important aspect of Theorem~\ref{th:result} is that it shows that the circumference (the length of a longest cycle) of an $n$-vertex, $D$-regular graph $G$ with $D \geq \alpha n$ cannot take values between roughly $(1-\alpha)n$ and $n-c$, where $c=c(\alpha) = 100\alpha^{-2}$. For our algorithm, this gives some slack to play with. On the other hand, for the Hamiltonicity problem, there is no such slack: by an easy generalisation of the example of Jung~\cite{Jung} and Jackson-Li-Zhu~\cite{JacksonLiZhu} (see Figure~\ref{fig:ex2}) there are regular graphs of degree roughly $n/k$ whose circumference is $n-(k-3)$.
	\begin{figure}
		\centering
		\includegraphics[width=0.4\textwidth]{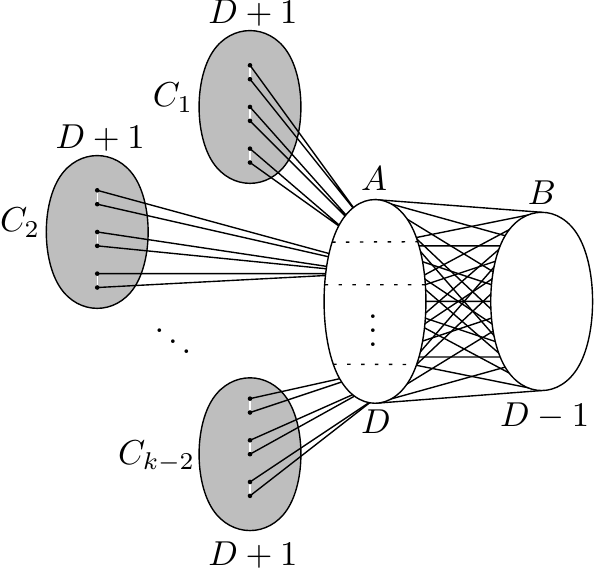}
		\caption{The graph $G$ above has $n = kD + k -3$ vertices (and we assume $k$ divides $D$ for simplicity). $A$ and $B$ are independent sets with all edges between them present. There are $D/k$ independent edges from $A$ to each $C_i$ so that these edges together form a matching. Then we delete a matching from each $C_i$ so that the resulting graph is $D$-regular.
				The graph has no cycle on $n- (k -4)$ vertices because deleting $D$ vertices from $G$ would then yield at most $D+ (k-4)$ components in $G$ (at most $D$ from the cycle and at most $k-4$ from the missed vertices), but deleting $A$ from $G$ yields $D + k -3$ components.}
		\label{fig:ex2}
	\end{figure}

If Hamiltonicity turns out to be NP-complete for dense, regular graphs then the question remains as to the smallest value of $c$ for which Theorem~\ref{th:result} holds. This may turn out to be closely related to the smallest $c$ for which the the circumference cannot take values between roughly $(1 - \alpha)n$ and $n-c$. It is also worth noting that the example in Figure~\ref{fig:ex2} has a large independent set  (roughly of size $\alpha n$) and one can in fact show that any non-Hamiltonian dense regular graph with long cycles (say of length at least $(1  - (\alpha/2))n$) must have a large independent set (of size at least $(\alpha - \varepsilon )n$).

Finally, we expect that the algorithm given in Theorem~\ref{th:result} can be modified to give an approximation algorithm for the longest path/cycle problems in dense regular graphs. The idea would be to search for (similarly to Lemma~\ref{lem:connect}) a connecting path system that maximises the number of vertices in the parts it connects together; write $S$ for this union of parts. We would then combine it with a balancing path system (guaranteed by Theorem~\ref{th:balance}) and use the resulting path system together with (a variant of) Lemma~\ref{lem:cycle} to produce a cycle passing through all but a fixed number $c$ of vertices in $S$. We should not expect any paths/cycles of length bigger than $|S|$ so this would give a $(1 - \frac{c}{n})$-approximation for the longest path/cycle. 
	

	

\section*{Acknowledgements}
We would like to thank Allan Lo for helpful discussions.

\bibliographystyle{mybibstyle}
\bibliography{writeup}

\begin{thebibliography}{10}

\bibitem{alon1986eigenvalues}
N. Alon.
\newblock Eigenvalues and expanders.
\newblock {\em Combinatorica}, 6(2):83--96, 1986.

\bibitem{Arora}
S. Arora, D. Karger, and M. Karpinski.
\newblock Polynomial time approximation schemes for dense instances of
  {NP}-hard problems.
\newblock {\em J. Comput. System Sci.}, 58(1):193--210, 1999.

\bibitem{AroraSTOC}
S. Arora, D.~R. Karger, and M. Karpinski.
\newblock Polynomial time approximation schemes for dense instances of
  \emph{NP}-hard problems.
\newblock pages 284--293. {ACM}, 1995.

\bibitem{bollobasconj}
B. Bollob\'as.
\newblock {\em Extremal Graph Theory}.
\newblock Academic Press, 1978.

\bibitem{christofides2012finding}
D. Christofides, P. Keevash, D. K{\"u}hn, and D. Osthus.
\newblock Finding hamilton cycles in robustly expanding digraphs.
\newblock {\em J. Graph Algorithms Appl.}, 16(2):335--358, 2012.

\bibitem{Csaba2002Approx}
B. Csaba, P. Krysta, and M. Karpinski.
\newblock Approximability of dense and sparse instances of minimum
  2-connectivity, tsp and path problems.
\newblock {\em Proceedings of the Thirteenth Annual ACM-SIAM Symposium on
  Discrete Algorithms (SODA-02), ACM}, pages 74--83, 02 2002.

\bibitem{Rob2}
B. Csaba, D. K\"{u}hn, A. Lo, D. Osthus, and A. Treglown.
\newblock Proof of the 1-factorization and {H}amilton decomposition
  conjectures.
\newblock {\em Mem. Amer. Math. Soc.}, 244(1154):v+164, 2016.

\bibitem{de1999approximation}
W.~F. De~La~Vega and M. Karpinski.
\newblock On the approximation hardness of dense tsp and other path problems.
\newblock {\em Information Processing Letters}, 70(2):53--55, 1999.

\bibitem{diestel}
R. Diestel.
\newblock Graph theory.
\newblock {\em Graduate texts in Mathematics}, 173, 2016.

\bibitem{dirac1952some}
G.~A. Dirac.
\newblock Some theorems on abstract graphs.
\newblock {\em Proceedings of the London Mathematical Society}, 3(1):69--81,
  1952.

\bibitem{Garey}
M.~R. Garey and D.~S. Johnson.
\newblock {\em Computers and intractability}.
\newblock W. H. Freeman and Co., San Francisco, Calif., 1979.

\bibitem{garey1976planar}
M.~R. Garey, D.~S. Johnson, and R.~E. Tarjan.
\newblock The planar hamiltonian circuit problem is np-complete.
\newblock {\em SIAM Journal on Computing}, 5(4):704--714, 1976.

\bibitem{CyclePartitions}
V. Gruslys and S. Letzter.
\newblock Cycle partitions of regular graphs, 2018.
\newblock arXiv preprint arXiv:1808.00851.

\bibitem{hopcroft2karp}
J.~E. Hopcroft and R.~M. Karp.
\newblock An $n^{5/2}$ algorithm for maximum matchings in bipartite graphs.
\newblock {\em SIAM J. Comput.}, 2(4):225--231, 1973.

\bibitem{jackson1980hamilton}
B. Jackson.
\newblock Hamilton cycles in regular 2-connected graphs.
\newblock {\em Journal of Combinatorial Theory, Series B}, 29(1):27--46, 1980.

\bibitem{JacksonLiZhu}
B. Jackson, H. Li, and Y.~J. Zhu.
\newblock Dominating cycles in regular {$3$}-connected graphs.
\newblock {\em Discrete Math.}, 102(2):163--176, 1992.

\bibitem{Jung}
H.~A. Jung.
\newblock Longest circuits in {$3$}-connected graphs.
\newblock In {\em Finite and infinite sets, {V}ol. {I}, {II} ({E}ger, 1981)},
  volume~37 of {\em Colloq. Math. Soc. J\'{a}nos Bolyai}, pages 403--438.
  North-Holland, Amsterdam, 1984.

\bibitem{kuhn2014robust}
D. K{\"u}hn, A. Lo, D. Osthus, and K. Staden.
\newblock The robust component structure of dense regular graphs and
  applications.
\newblock {\em Proceedings of the London Mathematical Society}, 110(1):19--56,
  2014.

\bibitem{kuhn2016solution}
D. K{\"u}hn, A. Lo, D. Osthus, and K. Staden.
\newblock Solution to a problem of bollob{\'a}s and on hamilton cycles in
  regular graphs.
\newblock {\em Journal of Combinatorial Theory, Series B}, 121:85--145, 2016.

\bibitem{Rob4}
D. K\"{u}hn, R. Mycroft, and D. Osthus.
\newblock A proof of {S}umner's universal tournament conjecture for large
  tournaments.
\newblock {\em Proc. Lond. Math. Soc. (3)}, 102(4):731--766, 2011.

\bibitem{Rob0}
D. K\"{u}hn and D. Osthus.
\newblock Hamilton decompositions of regular expanders: a proof of {K}elly's
  conjecture for large tournaments.
\newblock {\em Adv. Math.}, 237:62--146, 2013.

\bibitem{Robb2}
D. K\"{u}hn and D. Osthus.
\newblock Hamilton decompositions of regular expanders: applications.
\newblock {\em J. Combin. Theory Ser. B}, 104:1--27, 2014.

\bibitem{trevisan2012max}
L. Trevisan.
\newblock Max cut and the smallest eigenvalue.
\newblock {\em SIAM Journal on Computing}, 41(6):1769--1786, 2012.

\bibitem{trevisanLectureNotes}
L. Trevisan.
\newblock Lecture notes cs294, 2016.
\newblock Available at
  {https://lucatrevisan.wordpress.com/2016/02/09/cheeger-type-inequalities-for-$\lambda$n/}.

\end{thebibliography}
\section*{Appendix}

\begin{proof}[Proof (of Lemma~\ref{lem:cluster})]
	We define $f^*,f':(0,1)\to(0,1)$ as $f^*(x) = \min\{x^2/4, f(x)\}$, and
	$f'(x) = f^*_{5}(x)$, where $f^{*}_5(x)$ denotes composing $f^*$ with itself five times. Note that $f^*(x)<x$ and $f^*(x) \leq f(x)$ for all $x \in (0,1)$, so (by induction) $f^*_{5}(x)<f(x)$ for all $x \in (0,1)$. 

	We choose $\zeta,\delta,\gamma,\beta,\eta$ such that $\delta=f^*(\alpha), \zeta=f^*(\nu), \gamma=f^*(\zeta),\beta=f^*(\gamma),\eta=f^*(\beta)$. Note that this also implies $\tau\leq f^*(\delta)$ and $\rho\leq f^*(\eta)$.
Writing $x \ll_{f^*} y$ to mean that $x \leq f^*(y)$, one easily checks that
\[
\rho  \ll_{f^*} \eta  \ll_{f^*} \beta \ll_{f^*} \gamma \ll_{f^*} \zeta \ll_{f^*} \nu \leq \tau \ll_{f^*} \delta \ll_{f^*} \alpha.
\]	
Furthermore (D6) implies $m := k+\ell \leq 2n/D \leq 2 \alpha^{-1}$ and so $D/m \geq \alpha n / m \geq \alpha^2 n / 2 \geq \delta n$.
	
	Property (a) follows from (D2), (D3) and $\rho\ll\eta$. Property (b) follows from (D4) and $\alpha/m \geq 2\alpha^2 \geq \delta$.	
	
	For property (c), let $X,Y$ be a non-trivial partition of $A_i$. We will show $e_G(X,Y)> \zeta |X||Y|$.
	
	 First we consider the case that $A_i$ is a robust expander component. Assume without loss of generality that $|X| \leq |Y|$. 	
	If $|X|<\tau|A_i|$, each vertex in $|X|$ sends at least $D/m-|X|$ edges to $|Y|$ by (D4). Then
	$\frac{D}{m} -|X| \geq  \delta n - \tau|A_i| \geq  \zeta n \geq  \zeta |Y|$, so $e_G(X,Y) \geq  \zeta|X||Y|$. 
	If $|X|\geq \tau|A_i|$, then since $|X|\leq|Y|$, we have $|X|\leq |A_i|/2\leq (1-\tau)|A_i|$. Therefore 	$|\RN_{\nu,A_i}(X)|\geq |X|+\nu|A_i|$, so $|\RN_{\nu,A_i}(X)\cap Y|\geq \nu|A_i|$, and so $e_G(X,Y)\geq \nu^2|A_i|^2 \geq \zeta |X||Y|$.
	
	Now consider the case that $A_i$ is a bipartite robust expander component with parts $U_1$, $U_2$. Let $X$ be such that $|X\cap U_1| \leq  |Y\cap U_1|$, so we also have $|X\cap U_1|\leq |U_1|/2$.
	
	If $|X\cap U_1|<\tau|U_1|$ and $|X\cap U_2|<\tau|U_1|$, we  have 
	\begin{align*}
		e_G(X\cap U_1,Y\cap U_2)\geq& |X\cap U_1|(D/2m - |X\cap U_2|)\\ \geq& |X\cap U_1|(\delta n/2-\tau|U_1|)\geq \zeta n|X\cap U_1|.
	\end{align*}
	By the same argument $e_G(Y\cap U_1,X\cap U_2)\geq \zeta n|X\cap U_2|$, and together they sum up to $e_G(X,Y)\geq \zeta |X||Y|$.

	If $|X\cap U_1|<\tau|U_1|$ and $|X\cap U_2|\geq\tau|U_1|$, we have $e_G(Y\cap U_1,X\cap U_2)\geq (D/2m - |X\cap U_1|)|X\cap U_2|\geq 2\zeta n|X\cap U_2|\geq \zeta n|X|\geq \zeta |X||Y|$.
	
	If $|X\cap U_1|\geq \tau|U_1|$, then since $|X \cap U_1| \leq |Y \cap U_1|$, we have that 
	\[
	\tau|U_1| \leq |X \cap U_1|, |Y \cap U_1| \leq (1 - \tau)|U_1|.
	\]
	Therefore (dropping subscripts in $\RN$), 
	\begin{align}
	|\RN(X\cap U_1)\cap U_2|+|\RN(Y\cap U_1)\cap U_2| 
	&\geq  |U_1|+2\nu |A_i| \notag \\
	&\geq |U_2| +2 \nu |A_i| - \rho n \notag \\
	&\geq |U_2| + \nu|A_i|, \label{eq:UX}
	\end{align}
	using Proposition~\ref{Claim1}(i) and $\rho \ll \nu$ for the last inequality. 
	This implies that  $|\RN(X\cap U_1)\cap (Y\cap U_2)| > \nu |A_i|/2$ or  $|\RN(Y\cap U_1)\cap (X\cap U_2)| > \nu |A_i|/2$ since if both fail then we have
	\[
	|\RN(X\cap U_1) \cap U_2| < (\nu/2) |A_i| +|X \cap U_2|
	\text{ and }
	|\RN(Y\cap U_1) \cap U_2| < (\nu/2) |A_i| +|Y \cap U_2|,
	\] 
	 which when summed contradict \eqref{eq:UX}.
	   Without loss of generality, we assume $|\RN(X\cap U_1)\cap (Y\cap U_2)| > (\nu/2) |A_i|$, so that $e_G(X,Y)\geq  e_G(X\cap U_1,Y\cap U_2) \geq  \nu^2|A_i|^2/4\geq  \zeta |X||Y|$. 
	
	For property (d), if $A_i$ is a bipartite robust expander component with bipartition $U_1$, $U_2$ then 
	the number of non-$U_1$-$U_2$ edges is at most $e_G(U_1,\overline{U_2}) + e_G(U_2,\overline{U_1}) \leq  \rho n^2 \leq \beta n^2$, showing that $A_i$ is $\beta$-almost-bipartite with partition $U_1$, $U_2$. 
	If instead $A_i$ is a robust expander component, we claim that $A_i$ is $\gamma$-far from bipartite. Let $X,Y$ be a non-trivial partition with $|X| \leq |Y|$, so $|X|\leq |A_i|/2\leq (1-\tau)|A_i|$.
	If $|X|<\tau|A_i|$, then $e_G(X,Y)\leq |X|D$, so 
	\begin{align*}
	e(X)+e(Y)\geq (D/2m)|A_i| - D|X| \geq
	 \alpha n |A_i| ( (2m)^{-1} - \tau)
	  &\geq\ (\alpha^3 / 16) n^2 \\
	  &\geq \gamma |X||Y|,
	\end{align*}
	where the penultimate inequality follows since $|A_i| \geq \alpha n/2$ by (D3) and Remark~\ref{rem1}, and $m \leq k + 2\ell \leq 2\alpha^{-1}$ by (D6).
	If $|X|\geq \tau|A_i|$, then recalling $|X| \leq (1- \tau)|A_i|$, we also have $\tau |A_i| \leq |Y| \leq (1-\tau)|A_i|$ so $\RN_{\nu,A_i}(Y) \geq |Y| + \nu |A_i|$. Therefore, since  $|Y|\geq |A_i|/2$, we have  $|\RN(Y)\cap Y|\geq |Y| + \nu|A_i|$, so $e(Y)\geq  \nu^2|A_i|^2/2\geq \gamma |X||Y|$.	
\end{proof}
\end{document}